\documentclass[10pt,a4paper]{article}
\usepackage{amsmath, amsthm, amssymb,amsfonts,mathrsfs,amscd,latexsym}

\usepackage{hyperref}
\usepackage{url}

\usepackage[all]{xy}
\usepackage{colortbl}

\usepackage{color}
\usepackage{fancybox}
\usepackage{colordvi}
\usepackage{multicol}
\usepackage{colordvi}

\newtheorem{thm}{Theorem}[section]
\newtheorem{cor}[thm]{Corollary}
\newtheorem{lem}[thm]{Lemma}
\newtheorem{prop}[thm]{Proposition}

\theoremstyle{definition}
\newtheorem{defn}[thm]{Definition}

\newtheorem{rk}[thm]{Remark}
\newtheorem{ex}[thm]{Example}

\DeclareMathOperator{\Ker}{Ker}           
\DeclareMathOperator{\Op}{Op}      
\DeclareMathOperator{\Res}{Res}         
\DeclareMathOperator{\res}{res}         
 
\DeclareMathOperator{\spec}{sp}           
\DeclareMathOperator{\TR}{TR}                 

\newcommand{\A}{\mathcal{A}}              

\renewcommand{\a}{\alpha}                    
\newcommand{\C}{\mathbb{C}}              

\newcommand{\N}{\mathbb{N}}
\newcommand{\Z}{\mathbb{Z}}

\newcommand{\R}{\mathbb{R}}
\newcommand{\T}{\mathbb{T}}
\newcommand{\dbar}{d{\hskip-1pt\bar{}}\hskip1pt}
\newcommand{\cutoffint}{-\hskip -10pt\int}
\renewcommand{\t}{\mathbf{t}}

\newcommand {\e} {{\epsilon}}

\newcommand {\Ci} {{C^\infty}}

\renewcommand{\th}{\theta} 

\newcommand{\cutoffsum}{-\hskip -4mm\sum}

\newcommand{\W}{\underline{\mathcal W}}
\newcommand{\ellt}{\underline{\mathcal \ell^2\Gamma}}

\begin{document}

\title{\bf Traces of holomorphic families of operators on the noncommutative torus and on Hilbert modules}

\author{Sara Azzali, Cyril L\'evy, Carolina Neira, Sylvie Paycha}

%
%
%
%

\maketitle

\begin{abstract}
We  revisit traces of holomorphic families of pseudodifferential operators on a closed manifold in view of geometric applications. We  then transpose the corresponding analytic constructions to two different geometric frameworks;   the   noncommutative torus   and   Hilbert modules.  These traces are meromorphic functions whose residues at the poles as  well as  the constant term of the Laurent expansion at zero (the latter when the family at zero is a differential operator) can be expressed in terms of   Wodzicki residues and extended Wodzicki residues  involving logarithmic operators. They are therefore local and contain geometric information. For holomorphic families leading to  zeta regularised traces, they relate to the heat-kernel asymptotic coefficients via an inverse Mellin mapping theorem. We     revisit   Atiyah's $L^2$-index theorem  by means of the (extended) Wodzicki residue and   interpret the scalar curvature on the noncommutative two torus as an (extended) Wodzicki  residue.
\end{abstract}


 \section*{Introduction}
 
The canonical trace ${\rm TR}(A(z))$  of a holomorphic family  ${\mathcal A}\colon z\mapsto A(z)$ of classical pseudodifferential operators of affine order $-qz+a$, $q>0$ acting on smooth sections of a vector bundle over an $n$-dimensional closed manifold ($n>1$), defines a meromorphic function $z\mapsto {\rm TR}(A(z))$ with a  discrete set of simple poles $\{d_j=\frac{a+n-j}{q}, j\in \Z_{\geq 0}\}$. The residue at a pole $d_j$ is proportional to the Wodzicki residue of the operator $A(d_j)$.  These are well-known results due to Wodzicki whose residue \cite{Wo} is the only\footnote{when the dimension of the manifold is greater than one.} (up to a multiplicative factor) trace on classical pseudodifferential operators of integer order,  Guillemin \cite{Gu} who introduced the notion of   gauged symbols    and Kontsevich and Vishik \cite{KV} whose canonical trace TR corresponds to the unique   linear extension to operators of non integer order of the ordinary trace defined on operators of order $<-n$.
\\
When $A(0)$ is a differential operator, there is no pole at zero  since the residue vanishes on differential operators.  It was later observed  in \cite{PaSc} that  the limit $\lim_{z\to 0} {\rm TR}(A(z))$ is also proportional to a Wodzicki residue, namely to the  extended Wodzicki residue ${\rm Res}\left(A^\prime(0)\right)$, {\it extended} since the operator  $A^\prime(0)$ given by the derivative at zero  is typically not a classical operator any longer.  So the regularised trace $\lim_{z\to 0} {\rm TR}(A(z))$  is then local as a consequence of the locality of the Wodzicki residue.   
\\  If the meromorphic map  $\phi_{\mathcal A }\colon  z\longmapsto \Gamma(z)\,{\rm TR}\left(A (z)\right)$ corresponding to the holomorphic family ${\mathcal A}\colon z\longmapsto A(z)$ is the inverse Mellin transform of some function $f_{\mathcal A}:\R_+\longrightarrow  \R$, it follows from the inverse Mellin mapping theorem (see Proposition \ref{prop:InverseMellin}) that  $f_{\mathcal A}$ admits an asymptotic expansion    at $0$ given by
  $$f_{\mathcal A}(t)= \frac{1}{q}\sum\limits_{j\geq 0} a_j\left({\mathcal A}\right) t^{-d_j} +O(t^{-\gamma}),$$ for some appropriate $\gamma$. Its singular coefficients are Wodzicki residues (Theorem \ref{thm:mainthm})
  \begin{equation}\label{eq:aj-intro}
  a_j\left({\mathcal A}\right)= -\frac{1}{q} {\rm Res}\left( A(d_j)\right) \quad\text{for}\quad d_j>0,
  \end{equation} 
  and  its constant term \begin{equation}\label{eq:ajbis-intro}
  a_j\left({\mathcal A}\right)  =-\frac{1}{q} {\rm Res}\left(A^\prime(0)\right)\quad \text{for}\quad d_j=0, 
  \end{equation}  is an extended Wodzicki residue. 
\\ For holomorphic families ${\mathcal A}\colon z\longmapsto A(z)= A\, Q^{-z}$ built from complex powers $Q^{-z}$  of some appropriate invertible elliptic operator $Q$ and a classical pseudodifferential operator $A$, the canonical trace ${\rm TR}(A(z))$ is called the $\zeta$-regularised trace $\zeta(A,Q)(z)$ of $A$ with respect to the weight $Q$. It follows from the above discussion, that if $A(0)=A$ is a differential operator, then $\zeta(A,Q)(0)$ is a local quantity proportional to the extended Wodzicki residue $${\rm Res}\left(A\log Q\right)=-\frac{1}{q}\int_M {\rm res}_x( A\,\log Q)\, dx ,$$ where ${\rm res}_x( A\,\log Q)\, dx$ is the pointwise extended residue density involving the logarithm of $Q$. When $A=I$ is the identity operator, this is the logarithmic residue  ${\rm Res}\left( \log Q\right)$  investigated in \cite{O2} and \cite{Sc1}. \\
When   $Q=\Delta+\pi_\Delta$  with $\Delta$  an elliptic   
	differential operator to which we add the orthogonal projection $\pi_\Delta$ onto its kernel, making $Q$   invertible, then $A(z)= A\, Q^{-z}$ is the Mellin transform of the analytic family $\tilde A(t)= A\, e^{-t\left(\Delta+\pi_\Delta\right)}$. If $\Delta$ is nonnegative,  the inverse Mellin transform of the meromorphic map $  z\mapsto \zeta(A, Q)(z)$ is the Schwartz function
	 $$ f_{\mathcal A}\colon t\longmapsto {\rm Tr}\left(A\, e^{-t \left(\Delta+\pi_\Delta\right)}\right)=\int_M A\, K_t\left(\Delta+\pi_\Delta\right)(x,x)\,\sqrt{{\rm det}(g)}(x)\, dx
	 $$ 
	 on $]0,+\infty[$ where     $K_t\left(A\right)(x,x)$ denotes the fibrewise trace at the point $x$ of the  kernel of $e^{-t A}$ restricted to the diagonal and ${\rm det}(g)$ the determinant of a Riemannian metric $g$ on the manifold $M$.  
	\\  When applied to any multiplication operator $A=\phi$  given by a smooth function $\phi$ on $M$, this shows that both the constant {\it and}  the singular coefficients in the time zero asymptotic expansion of the kernel $K_t(\Delta)(x,x)$  \footnote{It follows from Duhamel's formula \cite{BGV} that the time zero asymptotic expansions of $K_t\left(\Delta+R\right)(x,x)$ and $K_t(\Delta)(x,x)$  coincide for every smoothing operator $R$ and hence in particular for $R=\pi_\Delta$, a fact that will be implicitly used throughout the paper. }can be expressed as  Wodzicki residues\footnote{Similarly, since the Wodzicki residue vanishes on smoothing operators, a perturbation of the operator by a smoothing operator so in particular by the projection onto the kernel does not modify the residue  (see Corollary \ref{cor:reslogxlocal}).}  (see Theorem \ref{thm:an}):
	\begin{eqnarray}\label{eq:HKintro} 
	&& K_t\left(\Delta \right)(x,x) \sim_{t\to 0}\\
	&&-\frac{(4\pi )^{\frac{n}{2}}}{2\, \sqrt{{\rm det} g}(x) }\,\Bigg[{\rm res}_x\left(\log \Delta \right) \, \delta_{\frac{n}{2}-\left[\frac{n}{2}\right]}\nonumber\\
	&&+\sum_{k\in\left[0,{\frac{n}{2}}\right[\cap \Z} \Gamma\left( \frac{n}{2}-k\right)  \,{\rm res}_x \left(   \Delta^{k-\frac{n}{2}}\right)\, t^{k-\frac{n}{2}}\Bigg]\nonumber.
	\end{eqnarray} 
	When integrated against  $\phi\in \Ci(M)$, the heat-kernel expansion (\ref{eq:HKintro}) yields the following heat-operator trace expansion at zero
	\begin{eqnarray}\label{eq:HKintrophi} 
	&&  {\rm Tr}\left(\phi \, e^{-t \Delta  }\right)\sim_{t\to 0}\\
	&&-\frac{(4\pi )^{\frac{n}{2}}}{2  }\,\Bigg[{\rm Res} \left(\phi\,\log \Delta \right) \, \delta_{\frac{n}{2}-\left[\frac{n}{2}\right]}\nonumber\\
	&+& \sum_{k\in\left[0,{\frac{n}{2}}\right[\cap \Z} \Gamma\left( \frac{n}{2}-k\right)  \,{\rm Res} \left( \phi\,   \Delta^{k-\frac{n}{2}}\right)\, t^{k-\frac{n}{2}}\Bigg]\nonumber.
	\end{eqnarray}
	That the singular coefficients of the heat-kernel (resp. heat-operator trace) expansion are proportional to the Wodzicki residue is a well-known (see e.g. \cite{Ac, KW}) fact often held for folklore knowledge. It  has been extended to noncommutative geometry \cite[Formula 1.5]{CC}) for the asymptotic expansion of the spectral action (take $f(\lambda)= e^{-t\lambda}$) whose non-constant coefficients arise as Dixmier traces.  What is lesser known, is  that
	\begin{itemize}
		\item  not only  the singular coefficients in the heat-kernel expansion (resp. heat-operator trace) but also the {\it constant  coefficient }   are  (possibly extended) Wodzicki residues (and hence   local),  a property that can   be easily transposed to other geometric frameworks in which  the canonical trace of holomorphic families of pseudodifferential operators can be built.
		\item    {\it no previous knowledge on the heat-operator trace asymptotics} is needed to express its coefficients in terms of Wodzicki residues, which is a purely analytic procedure.
	\end{itemize} 
	Since the Wodzicki residue of a logarithm is an algebraic expression involving the jets of the first $n$ homogeneous components of the symbol, we further recover other known facts (see \cite[Lemma 1.8.2]{Gi}), that the coefficients of the heat-operator trace expansion in  (\ref{eq:HKintrophi})  are
	\begin{itemize} 
		\item functorial algebraic expressions   in the jets of the homogeneous components of the  symbol of the operator $\Delta$ as a consequence of the corresponding property of the Wodzicki residue.
		\item Consequently,  if the operator $\Delta$ is of geometric nature, the coefficients are functorial algebraic expressions of   the jets of  the underlying metric and connection. 
	\end{itemize}     
Thanks to the functoriality of the construction, we can transpose our approach via inverse Mellin transforms to two different geometric contexts, applying it to
\begin{enumerate}
\item {\bf Holomorphic families of pseudodifferential operators on Hilbert modules:} Let $\Delta$ be an  (essentially) nonnegative selfadjoint  differential operator acting on a   vector bundle $E \to M$. One builds the operator $ \Delta_\mathcal H$ by twisting $\Delta$ by a flat connection on a bundle $\mathcal H$ of Hilbert modules over a finite von Neumann algebra. One can  then implement the same constructions as above to $\Delta_{\mathcal H}+R$ with $R$ a smoothing operator which makes the operator invertible (see Remark \ref{rk:inv}).  The  residue ${\rm Res}$ is then replaced by the  $\tau$-residue ${\rm Res}^\tau $ \eqref{eq:resAtau}, the $L^2$-trace (resp. canonical trace) ${\rm Tr}$ (resp. ${\rm TR}$) by ${\rm Tr}^\tau$ (resp. ${\rm TR}^\tau$), where $\tau$ is a finite trace on the von Neumann algebra. 

The corresponding heat-kernel  $\tau$-trace $K_t^\tau( \Delta_{\mathcal H} ) (x,x)$ at a point $x$ reads (see \eqref{eq:HKL2})
 \begin{eqnarray} \label{eq:HKintrolift} 
  && K_t^\tau  \left(\Delta_{\mathcal H}  \right)(   x,  x)  \\
 &\sim_{t\to 0} &-\frac{(4\pi )^{\frac{n}{2}}}{2 \sqrt{{\rm det}  g}(x) }
  \,\Bigg[ {\rm res}_x^\tau\left(\log \Delta_{\mathcal H} \right) \, \delta_{\frac{n}{2}-\left[\frac{n}{2}\right]} \nonumber \\
 	&+&
 	 \sum_{k\in\left[0, \frac{n}{2} \right[\cap \Z} \Gamma\left( \frac{n}{2}-k\right)  \,\ {\rm res}_x^\tau \left(   \Delta_{\mathcal H}^{k-\frac{n}{2}}\right) \, t^{k-\frac{n}{2}}\Bigg]\nonumber.
	  \end{eqnarray}
	 \begin{eqnarray*}
  &\sim_{t\to 0} & -\frac{(4\pi )^{\frac{n}{2}}}{2 \sqrt{{\rm det}  g}(x) }
  \,\Bigg[ {\rm res}_x \left(\log  \Delta \right) \, \delta_{\frac{n}{2}-\left[\frac{n}{2}\right]} \nonumber \\
 	&+&
 	 \sum_{k\in\left[0, \frac{n}{2} \right[\cap \Z} \Gamma\left( \frac{n}{2}-k\right)  \,\ {\rm res}_x  \left(  \Delta^{k-\frac{n}{2}}\right) \, t^{k-\frac{n}{2}}\Bigg]\nonumber.
 	 \end{eqnarray*}  
	 
	 The above formula \eqref{eq:HKintrolift} follows from the following more general property \eqref{eq:reslogEQ} relating the (extended) residue of  locally equivalent differential operators $A$, $A'$, respectively $B$, $B'$ (with some admissibility condition on the latter),
	 	\begin{eqnarray*}
&{\rm Res}^\tau\left(A \log B\right)={\rm Res}^\tau(A'  \log B')\\
&	{\rm Res}^\tau\left(A  B^\alpha \right)={\rm Res}^\tau\left(A'  B'^\alpha \right)
	,\;\;  \a\in \R,
\end{eqnarray*}
	where for simplicity we have dropped the projections onto the kernels since they are smoothing operators which are not ``seen" by the residue.
	   Atiyah's $L^2$-index theorem then boils down to an easy consequence, see Corollary \ref{cor:Aty} of the first of these two identities, using a $\mathbb Z_2$ graded version of the (extended) residue. This alternative proof, which is equivalent to Roe's heat equation proof \cite[Ch.15]{Ro}, is of pseudodifferential analytic nature and relies on the locality of the Wodzicki residue.

	\item   {\bf Holomorphic families of pseudodifferential operators on the noncommutative torus:}
	We consider  a  conformal perturbation $\mathbf{\Delta}_h$ (parameterised by a conformal factor $h$) of    a Laplace-type  operator  $\mathbf{\Delta}$  acting on the noncommutative $n$-torus $\T_\theta^n$,  where $\theta$ is an antisymmetric real matrix encoding the  noncommutativity. The  residue ${\rm Res}$ is then replaced by the  $\theta$-residue ${\rm Res}_\theta $ (see \cite{FW}), the $L^2$-trace (resp. canonical trace) ${\rm Tr}$ (resp. ${\rm TR}$) by ${\rm Tr}_\theta$ (resp. ${\rm TR}_\theta$, see \cite{LN-JP}). The coefficients of the Laurent expansions of traces of holomorphic families  can be  expressed in terms of Wodzicki residues (Theorem \ref{thm:NCKVPSopNCT})	   and  the heat-kernel expansion formula (\ref{eq:HKintrophi}) reads for any $a\in{\mathcal A}_\theta$, the Fr\'echet algebra of ``Schwartz functions" on the noncommutative torus $\T^n_\theta$ (see (\ref{eq:NCHKresphi}))
	\begin{eqnarray}\label{eq:HKintrophitheta} 
	&&  {\rm Tr}_\theta\left(a\, e^{-t \mathbf{\Delta}_h }\right)\sim_{t\to 0}\\
	&&-\frac{(4\pi )^{\frac{n}{2}}}{2  }\,\Bigg[{\rm Res}_\theta \left(a\,\log \mathbf{\Delta}_h \right) \, \delta_{\frac{n}{2}-\left[\frac{n}{2}\right]}\nonumber\\
	&+ &\sum_{k\in\left[0,{\frac{n}{2}}\right[\cap \Z} \Gamma\left( \frac{n}{2}-k\right)  \,{\rm Res}_\theta \left(a\,   \mathbf{\Delta}_h^{k-\frac{n}{2}}\right)\, t^{k-\frac{n}{2}}\Bigg]\nonumber.
	\end{eqnarray}
Going back to the setup of closed manifolds,  for certain geometric operators $\Delta$, the coefficients  of the heat-kernel expansion correspond to interesting geometric quantities; e.g. when $\Delta$ is the Laplace-Beltrami operator on a closed Riemannian manifold $(M,g)$, the coefficient of $t$ is proportional to the scalar curvature. The fact that the coefficient of $t$ in (\ref{eq:HKintrophitheta}) provides an analogue in the noncommutative setup of the scalar curvature on a noncommutative torus was exploited in \cite{CM1}, \cite{CM2}, \cite{CT}, \cite{FK1}, \cite{FK2}, \cite{FK3} to compute a noncommutative analogue of the scalar curvature on noncommutative tori.  By means of the Wodzicki residue on the noncommutative torus \cite{FW}, we use (\ref{eq:HKintrophitheta}) to define   the scalar curvature $\mathfrak s_h$ as a  Wodzicki residue; for any $a\in{\mathcal A}_\theta$ we set (compare with (\ref{eq:NCscal})
	\begin{equation}
	\langle \mathfrak s_h, a\rangle_h 
	= \begin{cases}
	-6\pi\,{\rm Res}_\theta \left(a\,\log   \mathbf{\Delta}_h \right)\quad  &\text{if}\quad n=2\noindent\\
	- \frac{ 3}{2}\, (4\pi )^{\frac{n}{2}} \, \Gamma\left( \frac{n}{2}-k\right)\, 
	{\rm Res}_\theta \left(a\,  \mathbf{\Delta}_h^{k-\frac{n}{2}} \right)\quad &  \text{otherwise},\nonumber
	\end{cases} 
	\end{equation}
	with $\langle\cdot,\cdot\rangle_h$ an adequate inner product on ${\mathcal A}_\theta$.
\end{enumerate}	
	 The extensions \eqref{eq:HKintrolift} and \eqref{eq:HKintrophitheta} are   possible thanks to the fact that  Cauchy calculus extends to Hilbert modules (see \cite{BFKM}) and to the noncommutative torus (see \cite{FW},\cite{LN-JP}). The analogies between the pseudodifferential calculi in these two frameworks leads to the question whether  the pseudodifferential calculus on $\mathcal N\mathbb Z^n$-Hilbert module encompasses the $\Z^n$-invariant pseudodifferential calculus on $\R^n$ studied in the more general context of global pseudodifferential calculus by Ruzhanski and Turunen in \cite{RT}. The latter corresponds to the $\theta=0$ case of the algebra $\Psi(\T^n_\theta)$  described in Section 4, which raises the further question, namely whether this issue can be carried out to the noncommutative setup.



\section[Traces of holomorphic families on closed manifolds]{ Traces of holomorphic families on closed manifolds and their geometry  }
\label{sec:1}

For the purpose of a later generalisation to the  Hilbert module setting, we recall the well-known basic setup to define the Wodzicki residue on closed manifolds and extend it to logarithmic pseudodifferential operators.

\subsection[The Wodzicki residue on pseudodifferential operators ]{The Wodzicki residue on  pseudodifferential operators }
\label{ssecres} 
Let $\left(M,g\right)$ be a closed Riemannian manifold of dimension $n$, let $p\colon E\to M$ be a vector bundle and let ${\rm End}(E)=E^*\otimes E $ be the corresponding endomorphism bundle.
\\ A  linear operator $A\colon  C^\infty\left(M, E\right)\to  C^\infty\left(M, E\right)$ is a (classical) pseudodifferential  operator of  order $a\in \C$, denoted $A\in \Psi^a(M, E)$, if in some atlas of $E  \to M$ it is of the form $A=\sum_{j=1}^J A_j+R$ where 
\begin{itemize}
	\item $R$ is a smoothing operator, namely a linear   operator $R\colon C^\infty\left(M, E\right)\to  C^\infty\left(M, E\right)$ with Schwartz kernel  given by a smooth  section of the bundle ${\mathcal L}(    M\times M, {\rm End}(E))$ of  linear bounded operators whose fibre at $(x,y)\in M\times M$ is the Banach space of bounded linear operators from the fibre  $E_y$ to the fibre $E_x$, 
	\item the operators $A_j$ are properly supported operators (meaning that the canonical projections $M\times M\to M$ restricted to the support of the Schwartz kernel are proper maps)  from $C^\infty \left(M, E\right)$ into  itself   
	such that in any coordinate chart in a neighborhood $U$ of a point $x\in M$ the operators $A_j$ are of   the form
	\begin{equation}    
	\label{eq: class.op}  u\longmapsto \frac{1}{(2\pi)^n}\int_{\R^n}d\xi \int_{U} dy\, e^{i\langle x-y,\xi\rangle}\, \sigma (x,\xi)\, u(y) 
	\end{equation}
	for some  symbol  $\sigma\in C^\infty \left(U
	\times \R^n, {\rm End}(E)\right)$  which is  asymptotically polyhomogeneous at infinity  \begin{equation}
	\label{eq:polyhom}\sigma(x,\xi)\sim\sum_{i=0}^\infty \omega(\xi)\, \sigma_{a-i}(x,\xi)
	\end{equation} i.e.,  the components 
	$\sigma_{a-i}\in C^\infty \left( T^*U\setminus  \left(U\times \{0\}\right) ,  {\rm End}(E)\right)$ being positively homogeneous of degree $a-i$ with $a$ the order of the operator. The asymptotic behaviour is to be understood as\footnote{The remainder depends on the choice of the cut-off function $\omega$ which is not explicitly mentioned here to alleviate the notation.}  
	\begin{equation} \label{eq:asymptsymb}\sigma^{(N)}(x,\xi):=\sigma(x,\xi)-\sum_{i=0}^{N-1}\omega(\xi)\, \sigma_{a-i}(x,\xi)\end{equation} is a symbol of order $\Re(a)-N$ for any $N\in \N$. Also, $\omega$ is  a smooth function  on $\R^n$ which is zero in a neighborhood of zero and identically one outside the unit ball.
	\\ 
	Let  us denote by ${\mathcal S}^a\left(U , {\rm End}(E)\right)$ the set of such symbols and by
	$${\mathcal S} \left(U ,{\rm End}(E)\right)=\langle \cup_{a\in \C}{\mathcal S}^a\left(U , {\rm End}(E)\right)\rangle $$ the algebra generated by all classical symbols of any complex order.\\
	Finally, let us denote by $\Theta(M,E)$ the subalgebra in $\Psi(M,E)$ of differential operators. 
\end{itemize} 
Let ${\rm tr}_x $ denote the fibrewise trace on ${\rm End}(E)$ above a point $x\in M$.  The {\bf pointwise Wodzicki  residue}   of   the operator $A$  with symbol $\sigma(A)$  at the point $x$ defined as 
\begin{equation}\label{eq:resx}\quad  {\rm res}_x (A) :=  \int_{|\xi|_x=1}{\rm tr}_x \left( \sigma(A)_{-n}(x,\xi)\right)\, \dbar_S\xi,
\end{equation}
vanishes on smoothing symbols so that it is independent of the choice of cut-off function $\omega$ chosen in (\ref{eq:polyhom}). Here     $d_S\xi$ is the measure on the unit cotangent sphere $S_x^*M=\{\xi \in T_x^*M, \, \vert \xi\vert_x:=g_x(\xi,\xi)=1\}$ induced by the one on the cotangent space $T_x^*M$ to $M$ at the point $x$ given by a Riemannian metric $g$ on $M$, $\dbar_S$ denotes the corresponding normalised measure $\dbar_S\xi=\frac{1}{(2\pi)^n}d_S\xi$. 
\begin{rk}
Clearly the residue  ${\rm res}_x $    vanishes on differential and  non integer order operators.
\end{rk} 
An important result of Wodzicki is   that $ {\rm res}_x  (A)\, dx$ defines a global density \cite{Wo}.
The {\bf Wodzicki residue} of the operator $A$ is then defined by 
\begin{equation}
\label{eq:resA}{\rm Res}(A)=\int_M {\rm res}_x(A)\, dx=   \int_M dx\int_{\vert \xi\vert_x=1} {\rm tr}_x\left(\sigma_{-n}(A)(x,\xi)\right)\,  \dbar_S\xi.
\end{equation} 

\subsection[Logarithms of pseudodifferential operators on closed manifolds]{Logarithms of pseudodifferential operators on closed manifolds}
\label{sseclog}

\begin{defn}\label{def:agmon} (see e.g. \cite{Se}, \cite{Sc2}) 
	Let $A$ be an  operator in $ \Psi(M, E)$.
	A real number $\beta $ is a principal angle of $A$ if there exists a ray $R_\beta= \{re^{i\beta}, \quad r\geq 0\}$ which is disjoint from the spectrum of the ${\rm End}(E_x)$-valued leading symbol $\sigma_L(A)(x,\xi)$ for any $x\in M$, $\xi \in T_x^*M \setminus \{0\}$. 
\end{defn}
\begin{defn}\label{def:admiss1}
	We call an operator $A\in \Psi(M, E)$   {\bf admissible} with {\bf spectral cut} $\beta$ if  
	\begin{itemize} 
		\item its order  is positive,
		\item   $\beta$ is a principal angle for $A$. 
	\end{itemize}
\end{defn}
\begin{rk}  
		An admissible operator is elliptic but not necessarily  invertible.
		Admissibility is a covariant condition, that is to say it is preserved under diffeomorphisms. 
	 \end{rk}
	\begin{defn} \label{defn:weight} We  call {\bf weight} an invertible admissible operator $Q\in \Psi(M, E)$.
	\end{defn} As we shall see later,  weights are used to regularise traces. Here is a useful lemma which strongly uses the theory of elliptic operators on closed manifolds.  
	\begin{lem}\label{lem:agmon angle}  If $A\in \Psi(M, E)$ is an admissible operator   with  spectral cut $\beta$ then there is an angle $\beta'$ arbitrarily close if not equal to $\beta$  and a truncated solid angle   $\Lambda_{\beta', \epsilon}:=\{z\in \C\setminus 0: \vert z\vert>\e,\, \arg z\in (\beta'-\epsilon,\beta'+\epsilon)\}$ for some  $\e>0$, outside of which lies the spectrum  of $A$.
		\\ If $A$ is moreover invertible i.e., if  it is a weight, then  its spectrum lies outside the solid angle $V_{\beta', \epsilon}:=\{z\in \C\setminus 0\, : \arg z\in (\beta'-\epsilon,\beta'+\epsilon)\}$ for some small $\e>0$. The angle $\beta'$ is then called an {\bf Agmon angle}. 
	\end{lem}
	\begin{proof}   Since $A$ is admissible, it is elliptic.   Since  the order of $A$ is positive,  the manifold $M$ being closed, the operator  has a purely discrete spectrum, which consists of countably many eigenvalues with no  accumulation point. So   if the spectrum  of $A$ meets the ray $R_\beta$,   there is a small perturbation $\beta'$ of $\beta$ such that the spectrum of $A$ does not meet $R_{\beta'}\setminus\{0\}$. Being discrete, the spectrum of $A$  actually lies outside   a truncated solid angle   $ \{z\in \C\setminus 0: \vert z\vert>\e,\; \arg z\in (\beta'-\epsilon,\beta'+\epsilon)\}$ for some  $\e>0$ chosen small enough so that  $0$ is  the only eigenvalue  in the ball of radius $\e$ centered at zero.  If moreover $A$ is invertible, then   its spectrum    lies outside   the  solid angle   $ \{z\in \C\setminus 0 : \arg z\in (\beta'-\epsilon,\beta'+\epsilon)\}$  and $\beta'$ is   an Agmon angle for $A$.  
	\end{proof}
	Let now $A$ be a weight in $\Psi^a(M,E)$ with spectral cut $\beta$. 
	Then for  $\Re(z)>0$, its complex powers  (see \cite{Se} for further details)
	\begin{equation}\label{eq.cxp}
	A_\beta^z= \frac{i}{2\pi} \int_{\Gamma_\beta} \lambda^z \, (\lambda-A)^{-1}\, d\lambda, \end{equation} 
	and respectively the operators \cite[Par. 2.6.1.2.]{Sc2}
	\begin{equation}\label{eq.Log}L_\beta(A, z)= \frac{i}{2\pi} \int_{\Gamma_\beta} \log_\beta \lambda \;\;\lambda^z \, (\lambda-A)^{-1}\, d\lambda, 
	\end{equation} 
	are bounded linear maps from any Sobolev closure $H^s\left(M, E\right)$ of $\Ci(M, E)$, $s\in \R$,  
	with values in $H^{s-a\Re(z)}\left(M, E\right)$ respectively  in $H^{s-a\Re(z)+\epsilon}\left(M, E\right)$, for any $\epsilon >0$. 
	Here $\Gamma_\beta$ is  a closed contour  in $\C\setminus \{re^{i\beta },\, r\geq 0\} $ around  the spectrum of $A$ oriented clockwise. These definitions extend to the whole complex plane
	$$A_\beta^z=A^k \;A_{\beta}^{z-k} \;\;\;, \;\;\text{resp.} \;\;L_\beta(A, z)=A^k\;  L_\beta(A, z-k), \qquad \Re(z)<k
	$$
	for any $k\in \N$, 
	$A_\beta^z$ is an operator in $\Psi(M,E) $ of order $az$ for any complex number $z$, and the logarithm 
	$$\log_\beta (A):=L_\beta(A, 0)
	$$ of $A$ is a bounded linear map from  $H^s\left(M, E\right)$ to $H^{s+\epsilon}\left(M, E\right)$, $\forall\epsilon >0$.
	One has by construction  $ \log_\beta A\,  A_\beta^{z}=
	A_\beta^{z}\, \log_\beta A$, $\forall z\in \C$.
	\begin{rk} 
		\label{rk:beta}Just as a complex power  does, the  logarithm depends on the choice of a
		spectral cut $\beta$. However, in order to simplify the notation we shall often drop
		the explicit mention of $\beta$. 
	\end{rk}
	The logarithm  of a classical pseudodifferential operator of positive order is   not classical. 
		Indeed, in a local trivialisation, the symbol of $\log_\beta A$ reads (see e.g.  \cite{Sc2})
	\begin{equation}\label{eq:symbollog}
	\sigma(\log_\beta A)( x,\xi)=a\, \log \vert \xi\vert I +
	\sigma_{\rm cl}(\log_\beta A)(x, \xi)
	\end{equation}
	where $a$ denotes the order of $A$ and $\sigma_{\rm cl}(\log_\beta  A)$ is a classical symbol of
	order zero with  homogeneous components   $\sigma_{-j}(\log_\beta   A)$ of degree $-j, j\in \Z_{\geq 0}$.
	\\    Moreover, the leading symbol $\sigma^L_{\rm cl}(\log_\beta A)$ of $\sigma_{\rm cl}(\log_\beta A)$ can be expressed in terms of the leading symbol $\sigma^L(A)$  of $A$ as
	\begin{equation}\label{eq:leadinglogsymb}\sigma^L_{\rm cl}(\log_\beta A)(x, \xi)=\log_\beta
	\left(\sigma^L(A)\left(x,\frac{\xi}{\vert
		\xi\vert}\right)\right)\quad \forall ( x,\xi)\in T^*M\setminus M\times\{0\}.\end{equation}

	\subsection[The local Wodzicki residue extended to logarithms]{The local Wodzicki residue extended to logarithms} 
	\label{ssecreslog}
	In spite of the fact that logarithms are not classical, the Wodzicki residue  on classical pseudodifferential operators does extend  to logarithms.\footnote{This extended residue  differs from the higher residue on log-polyhomogeneous operators introduced in \cite{L}.} As a first step  we extend to logarithms the local residue ${\rm res}_x$ defined in \eqref{eq:resx}.   
	For a weight $Q\in \Psi(M , E)$   with spectral cut $\beta$ we define the pointwise extended residue as
	$${\rm res}_x\left(\log_\beta Q\right):=\int_{|\xi|_x=1}{\rm tr}_x \left(\sigma_{-n}(\log_\beta Q)(x,\xi)\right)\, \dbar_S\xi.
	$$ 
	By \eqref{eq:symbollog} this is a natural extension of the pointwise residue on classical symbols since the integral over the sphere  vanishes on the logarithm of the norm. The fact that the $n$-form ${\rm res}_x\left(\log_\beta Q\right)\, dx $ defines a volume density will arise later as a consequence of a local formula for the $\zeta$-regularised  trace.

	The subsequent proposition   shows the locality of the extended residue  of $A\,\log Q$ in so far as it only depends on a finite number of homogeneous components of the symbols of $A$ and $Q$
	\begin{prop}\label{prop:reslogxlocalexpression}
		For any   weight $Q\in \Psi(M,E)$ of order $q$, 
		\begin{enumerate}
			\item (compare with \cite[Lemma  1.8.2]{Gi})\\
			the local logarithmic residue  at a  point $x\in M$  which reads
			\begin{equation} \label{eq:resxlogQ} 
			{\rm res}_x(\log Q)=  \frac{1}{ 2\pi i } \int_{\vert \xi\vert_x=1}\, \left( \int_\Gamma \log \lambda\, \sigma_{-q-n}
			( Q-\lambda)^{-1 }(x,\xi)\, d\lambda \right) \, \dbar_S\xi,   
			\end{equation}
			where $\Gamma$ is a contour around the spectrum of $Q$ oriented clockwise, is an algebraic expression  in the  $x$-jets  of the first $n$ homogeneous components  (taken in decreasing order of homogeneity) of the symbol $\sigma(Q)(x, \cdot) $ of $Q$ at that point given by the integral over the unit cotangent sphere of  an algebraic expression  in the  $(x,\xi)$-jets  of the first $n$   homogeneous components  of the symbol $\sigma(Q)(x, \cdot) $ of $Q$ at that point. Here  $\sigma_{-q-n}( Q-\lambda)^{-1 }$ is the $(-q-n)$--th homogeneous component of the resolvent $( Q-\lambda)^{-1 }$ of $Q$.
			\item (compare with \cite[Lemma 1.9.1]{Gi})\\
			Given a   differential operator $A=\sum_{\vert \alpha\vert \leq a} a_\alpha(x) D_x^\alpha \in \Theta(M,E)$ of order $a\in \Z_{\geq 0}$, then  ${\rm res}_x\left(A\log Q\right)$ is an algebraic expression  in the  coefficients $a_\alpha$ of $A$ and in the $x$-jets  of   the first $n+a$ homogeneous components  of the symbol $\sigma(Q)(x, \cdot) $ of $Q$ at that point.   
			
			In particular, if $Q=\sum_{\vert \beta\vert \leq q} b_\beta(x)D_x^\beta$ is a differential operator, then the local residue  $\res_x (A\,\log (Q)) $ is  an algebraic expression 
			in  the coefficients $a_\alpha$ and in  the $x$-jets of the coefficients $b_\beta$.
		\end{enumerate}
	\end{prop}
	
	\begin{proof}
		Since
		\begin{eqnarray*}  
			\sigma_{-n} (A\,\log  Q)
			&= &\sum_{\vert\gamma\vert+j+k=a+n}\frac{(-i)^\gamma}{\gamma!} \partial_\xi^\gamma\sigma_{a-j}(A)\,\partial_x^\gamma\sigma_{-k} (\log  Q)\\
			&=& \sum_{\vert\gamma\vert + k=  a+n} (-i)^\gamma {\alpha\choose \gamma}  a_\alpha(x)\xi^{a-\gamma}\,  \partial_x^\gamma\sigma_{-k} (\log  Q),
		\end{eqnarray*}
		the fact that ${\rm res}_x\left(A\log Q\right)$   is an algebraic expression  in the  coefficients $a_\alpha$ of $A$ and in the $x$-jets  of   the first $n+a$ homogeneous components  of the symbol $\sigma(Q)(x, \cdot) $   follows from  a similar statement for   the homogeneous components $\sigma_{-k} (\log  Q)$ of order $-k\in \Z_{\leq 0}$   of the symbol of $\log Q$.   Now, 
		since $\log Q= \partial_z Q^z\vert_{z=0}$, the homogeneous component $\sigma_{-k}(\log Q)$  is derived by differentiating the homogeneous component $\sigma_{qz-k}( Q^z)$ of degree $qz-k$  at zero   of  the complex power $Q^z$. The latter is obtained from the homogeneous component $\sigma_{-q-k}$ of the resolvent $(Q-\lambda)^{-1}$ by means of the Cauchy formula
		$$ Q^z=\frac{1}{2i\pi  } \int_\Gamma \lambda^z ( Q-\lambda)^{-1 } d\lambda$$ where, as before, $\Gamma$ is a contour around the spectrum of $Q$ oriented clockwise. Thus we find
		$$\sigma_{-k}(\log Q)=  \frac{1}{2i\pi  } \partial_z \left(\int_\Gamma \lambda^z \sigma_{-q-k}( Q-\lambda)^{-1 }\, d\lambda\right)_{\vert_{z=0}}.$$ In particular, setting $k=n$ and integrating over the unit cotangent sphere we find
		$$
		{\rm res}_x(\log Q)=  \frac{1}{2\pi i} \int_{\vert \xi\vert_x=1}\partial_z \left(\int_\Gamma \lambda^z \sigma_{-q-n}( Q-\lambda)^{-1 }\, d\lambda\right)_{\vert_{z=0}}\, \dbar_S\xi,
		$$ 
		which yields (\ref{eq:resxlogQ}).   Since $\sigma_{-q-k}( Q-\lambda)^{-1 }(x,\xi)$ is an algebraic expression in the $(x,\xi)$-jets of the first $k$ homogeneous components $$\sigma_{q-1}(Q)(x,\xi), \sigma_{q-2}(Q)(x,\xi), \cdots, \sigma_{q-k}(Q)(x,\xi)$$ of the symbol of $Q$, so is $\sigma_{-k}(\log Q)(x,\xi)$ an algebraic expression in the jets of the first $k$ homogeneous components. This for $n=k$   yields   the first assertion (compare with Formula (8) in \cite{MP}). The second assertion follows in a similar way after implementing the differential operator $A$ and integrating over the unit cotangent sphere.
	\end{proof} 
	
	The locality of the extended residue can also be seen from the fact that it does not detect smoothing perturbations.
	\begin{cor}\label{cor:reslogxlocal}
		Let $A\in \Psi(M, E)$ be an admissible operator. Let  $R, S$  be two smoothing operators acting on $\Ci(M,E)$ such that the perturbed operators $ A+R$ and $A+S$ are   invertible. They  define weights  and  for any $x\in M$ we have
		\begin{equation}\label{eq:resxlog}
		{\rm res}_x\left(\log (A+R)\right) = {\rm res}_x\left(\log (A+ S)\right).
		\end{equation}
	\end{cor}
	\begin{proof}
		This follows from Proposition \ref{prop:reslogxlocalexpression} and the fact that a smoothing perturbation  of an operator does not modify the homogeneous components of its symbol.
	\end{proof}
	Let  $\Delta\in \Theta(M,E)$  be  an admissible operator. 
		Let $E$ be equipped with a hermitian metric, which combined with a Riemannian metric on $M$ induces an inner product on $\Ci(M,E)$. It follows from the theory of elliptic operators on a closed manifold (see e.g. \cite{Gi}) that the orthogonal projection $\pi_\Delta$ onto the kernel Ker $(\Delta)$ is  a finite rank operator and hence smoothing. Consequently,   the operator $ Q:= \Delta +\pi_\Delta$   is a weight.
		On the grounds of Corollary \ref{cor:reslogxlocal},   we define the {\bf pointwise logarithmic residue} of $\Delta$ as 
		\begin{equation}\label{eq:reslogARX}
		\res_x(\log \Delta):=\res_x(\log (\Delta+\pi_\Delta))\quad \forall x\in M. 
		\end{equation} As we shall see below, the local density $\res_x( A\, \log \Delta)\, dx$ actually defines a global density on the manifold, confirming the known fact that the residue extends to logarithms \cite{O1,O2, PaSc}.

	\subsection[The Wodzicki residue as a complex residue]{The Wodzicki residue as a complex residue}
	
	Given a symbol $\sigma(x,\xi)\in {\mathcal S}^a\left(U,{\rm End }(E) \right)$ with $x$ a point in $M$ and  $U\subset M$ an open  neighborhood of $x$,  
	the  {\bf cut-off integral}\footnote{ also called    {\bf Hadamard finite part integral} see e.g. \cite[Example 2 chapter II]{Sc2}, and also \cite{Ge}.} is defined as  the finite part
	\begin{equation}\label{eq:cutoff} 
	\cutoffint_{\R^n}{\rm tr}_x  \sigma(x,\xi)\, \dbar\,\xi: = {\rm fp}_{R\to  \infty} \int_{B(0, R)} {\rm tr}_x \sigma(x,\xi) \, \dbar\,\xi
	\end{equation}
	
	\begin{rk} Whereas the residue vanishes on symbols whose order has real part smaller than $-n$, the cut-off integral coincides on those symbols with the ordinary  integral on $\R^n$.   A straightforward computation  shows  that, like the local residue, the cut-off integral also vanishes on polynomial symbols. \end{rk}
	
	We   need   holomorphic families of classical pseudodifferential symbols   first introduced by Guillemin in \cite{Gu} and extensively used by Kontsevich and Vishik in \cite{KV}.  The idea is to embed a symbol $\sigma $ in a family $z\mapsto \sigma(z) $ depending holomorphically on a complex parameter $z$.\\
	
	\begin{defn}\label{defn:holsymb}  Let $U$ be an open subset of $M$. We call  a   family   $ \left( \sigma(z) \right)_{z\in \Omega}$ of symbols in  ${\mathcal S} \left(U , {\rm End}(E)\right)$  parametrised by a domain $\Omega$ of $\C $ {\bf holomorphic} at a point $z_0\in \Omega$ if,  with the notation of \eqref{eq:polyhom} and (\ref{eq:asymptsymb}) we have  \begin{enumerate}
			\item $\sigma(z)(x,\cdot)$ is uniformly in $x$ on any compact subset of $U$, holomorphic at $z_0$ as a function of $z$ with values in $C^\infty \left(U\times\R^n,{\rm End}(E)\right)$,
			\item for any $z$ in  a neighborhood of $z_0$ there is an asymptotic expansion of the type (\ref{eq:polyhom})
			\begin{equation}\label{eq:logclassical}
			\sigma(z)(x,\cdot)\sim \sum_{j\geq 0}
			\sigma_{\alpha(z)-j}(z)(x,\cdot),
			\end{equation} 
			with $\alpha(z):=-qz+a$ for some positive number $q$ and $a$ the order of $\sigma:=\sigma(0)$,
			\item for any integer $N\geq 1$ the remainder
			$$\sigma_{(N)} (z):= \sigma(z)- \sum_{j=0}^{N-1}
			\sigma_{\alpha(z)-j}(z)$$ is  uniformly in $x$ on any compact subset of $U$,  holomorphic at $z_0$ as a function of $z$ with
			values in $C^\infty \left(U\times\R^n,{\rm End}(E)\right)$  with $k$--th $z$-derivative
			\begin{equation}\label{e:kthderivlogclassical}
			\sigma^{(k)}_{(N)}(z) := \partial_z^k(\sigma_{(N)}(z))
			\end{equation}
			a symbol on $U$ of order $\alpha(z)-N + \e$ for any $\e>0$  uniformly in $x$ on any compact subset of $U$ and  locally uniformly in $z$ around $z_0$, i.e. the $k$-th derivative $ \partial_z^k\sigma_{(N)}(z) $ satisfies a local uniform estimate in $z$ around $z_0$
			\begin{equation}\label{eq:symbolestimate}\Vert\partial_\xi^\beta  \partial_z^k\sigma_{(N)}(z)(x,\xi)\Vert\leq
			C_\beta \, \langle \xi\rangle ^{\Re(qz)-N-\vert
				\beta\vert}\quad \forall \xi\in \R^n,\end{equation} where  $\Vert A\Vert:= \sqrt{{\rm tr}_{x}(A^*A)}$ is the norm on ${\mathcal L}(    M\times M, {\rm End}(E))$ and where we  have set $\langle \xi\rangle:= \sqrt{1+\vert \xi\vert^2}$ with $\vert \cdot \vert$
			the Euclidean norm of $\xi$.
		\end{enumerate} 
	\end{defn}
	\begin{ex} If $\sigma\in S(U,{\rm End}(E))$ is a symbol of order $\alpha(0)$, then $\sigma(z)(x, \xi)=\sigma(x, \xi)\, \langle \xi\rangle^{-z}$ defines a holomorphic family of order  $\alpha(z)=-z+\alpha(0)$.
	\end{ex}
	The following assertion can be shown  on direct inspection of the cut-off integral.
	\begin{prop}\label{prop:KVPSsym}  For any holomorphic family $ \sigma(z)$ of classical symbols  parametrised by $  \C$ with affine order $\alpha(z)=-qz+a$ for some positive real number $q$ and some real number $a$,
		\begin{enumerate}
			\item  the map $$z\mapsto \cutoffint_{\R^n} \sigma(z)(x, \xi)
			\dbar\xi$$ is meromorphic with simple poles $d_j:=\frac{ a+n-j}{q},\quad j\in \Z_{\geq 0}$.
			\item $\,${\rm\cite{KV}} The complex  residue at the point $d_j$ in $ \C$  is given by:
			\begin{equation}\label{eq:classicalKV}{\rm Res}_{z=d_j}\left( \cutoffint_{\R^n}
			{\rm tr}_x\sigma(z)( x,\xi)\, \dbar\xi
			\right)= \frac{1}{q}{ \rm
				res}_x (\sigma(d_j)  ).
			\end{equation} 
			\item  $\,${\rm\cite{PaSc}}  The
			finite part
			at  the pole  $d_j$ differs from the cut-off regularised integral $
			\cutoffint_{\R^n} {\rm tr}_x\sigma (d_j)(x, \xi)\,
			\dbar\xi$ by
			\begin{equation}\label{eq:PSclassicalsymb}
			{\rm fp}_{z=d_j} \left(\cutoffint_{\R^n} {\rm tr}_x\sigma(z)(
			x, \xi)\, \dbar\xi\right) -\cutoffint_{\R^n} {\rm tr}_x\sigma(d_j)(x, \xi) \, \dbar\xi=  
			\frac{1}{q}{ \rm
				res}_x(\sigma^\prime (d_j)).
			\end{equation} Here    the noncommutative residue  is extended to the possibly
			non-classical
			symbol \footnote{The asymptotic expansion of
				$\tau_j  (x,\xi)$  as
				$\vert \xi\vert \to \infty$ might present logarithmic terms $\log\vert\xi\vert$,
				which vanish on the unit sphere and therefore do not explicitly arise in the
				following definition.} $\tau_j  (x,\xi):=\sigma^\prime(d_j)(x,\xi)$ using  the same formula as in Equation
			(\ref{eq:resx})
			$${ \rm
				res}_x (\tau_j ):= \int_{|\xi|_x=1}{\rm tr}_x \left(\tau_j\right)_{-n} ( x,\xi)\,
			\dbar_S\xi.$$
		\end{enumerate} 
	\end{prop}
	We are now ready to introduce holomorphic families of pseudodifferential operators. 
	
	\begin{defn}\label{defn:holop}
		Following \cite[Definition 1.14]{PaSc} we call  a family   $ \left( A(z) \right)_{z\in \Omega}$ of operators in  $\Psi \left(M , E\right)$  parametrised by a domain $\Omega$ of $\C $ holomorphic at a point $z_0\in \Omega$ if, with the notation of Section \ref{ssecres}, in each local trivialisation $U$ of $E\to M$ we have 
		\[
		A(z)=\sum_{j=1}^J A_j(z)+R(z)                                                                                                                                                                                                                                                                                                                                                                                                                                                                       \]
		with
		\begin{enumerate}
			\item $A_j(z)={\rm Op}(\sigma_j(z))$, where $\sigma_j(z)$ is a holomorphic family of polyhomogeneous symbols on $U$,
			\item $R(z)$ is a smoothing operator with Schwartz kernel $R(z,x,y)\in C^\infty(\Omega\times U\times U,{\rm End}(E))$ holomorphic in $z$.
		\end{enumerate}
	\end{defn} 
	Integrating the results of Proposition \ref{prop:KVPSsym} over $M$ yields the following theorem which we quote without proof, referring the reader to \cite{KV} and \cite{PaSc}.   Let us however recall that the linear map ${\rm TR}$ introduced in the theorem below is the canonical trace popularised in \cite{KV} i.e., the unique linear form (up to a multiplicative factor) on the subset of  $  \Psi(M,E)$ consisting of non integer order classical pseudodifferential operators,  which vanishes on commutators that lie in this set. It extends to differential operators where it vanishes and it coincides with the $L^2$-trace Tr on trace-class operators i.e. the real part of the order is smaller than $-n$. 
	
	\begin{thm}\label{thm:KVPSop}  For any holomorphic family  $ A(z)\in \Psi(M, E)$ of classical operators parametrised by $   \C$ with holomorphic order $   -qz+a$ for some positive $q$ and some real number $a$,
		\begin{enumerate}
			\item  the meromorphic map $z\mapsto {\rm TR}_x(A(z)):=\cutoffint_{\R^n} {\rm tr}_x\sigma(z)(
			x, \xi) \dbar\xi$ integrates over $M$ to  the map $$z\mapsto {\rm TR}\left(A(z)\right):= \int_M {\rm TR}_x(A(z))\, dx $$ which is  meromorphic with simple poles $d_j:=\frac{ a+n-j}{q},\quad j\in \Z_{\geq 0}$.
			\item $\,${\rm\cite{KV}} The complex  residue at the point $d_j$   is given by:
			\begin{equation}\label{eq:classicalKV2}{\rm Res}_{z=d_j} {\rm TR}\left(A(z)\right)= \frac{1}{q}{ \rm
				Res} (A(d_j) ).
			\end{equation} 
			\item  $\,${\rm\cite{PaSc}}  If $A(d_j)$   differs from a differential operator\footnote{This yields another application of the results of \cite{PaSc} since only the case  $A(d_j)$ differential was considered in the examples given in that paper.}    by a trace-class pseudodifferential operator $T_j$, then $A(d_j)$ has a well-defined canonical trace ${\rm TR}(A(d_j))={\rm Tr}(T_j)$   and  $A^\prime (d_j)$ has a well defined 
			Wodzicki residue $${ \rm
				Res} (A^\prime(d_j) ):=\int_M {\rm res}_x\left(A^\prime (d_j)\right)\, dx$$
			where
			\[
			{\rm res}_x\left(A^\prime (d_j)\right):=\int_{\vert \xi\vert_x=1} \sigma_{-n}\left(A^\prime (d_j)\right)(x,\xi)\, \dbar_S\xi
			\]
			at the point  $d_j$    and we have  
			\begin{equation}\label{eq:PSclassicalop}
			{\rm fp}_{z=d_j}{\rm TR}\left(A  (z)\right)=  {\rm Tr}(T_j)-
			\frac{1}{q}{ \rm
				Res}(A^\prime (d_j)).
			\end{equation} 
		\end{enumerate} 
	\end{thm}
	
	\begin{rk} Formula  \eqref{eq:PSclassicalop}   formally
		follows from  \eqref{eq:PSclassicalsymb}  applied
		to the family
		$\tau_j(z)=\frac{\sigma(z)-\sigma(d_j)}{z-d_j}$ since
		$\tau_j(d_j)=\sigma^\prime(d_j)$. However,   $\tau_j(z)$   not strictly speaking being a holomorphic family of classical
		symbols   since the two symbols
		$\sigma(z)$ and   $\sigma(d_j)$  have different orders outside $d_j$, which do not differ by an integer, the proof is actually slightly more indirect.\\
		Consequently, ${ \rm Res}(A^\prime (d_j))$ is  the noncommutative residue  extended   to the typically non-classical
		operator  $A^\prime (d_j)$.  
	\end{rk}
	\subsection[$\zeta$-regularised traces]{$\zeta$-regularised traces}

	Given a weight $Q\in \Psi(M,E)$ of order $q\in \R_+$ and an operator   $A\in \Psi(M,E)$ (not necessarily admissible) of order $a\in \R$, the map $z\mapsto A(z):= A\, Q^{-z} $  defines a holomorphic family. The subsequent theorem  quoted from \cite{PaSc} follows from applying Theorem \ref{thm:KVPSop} to this family. 
	
	\begin{thm}\label{thm:zetaresM}
		Given a weight $Q\in \Psi(M,E)$ of order $q\in \R_+$ and an operator   $A\in \Psi^{a}(M,E)$, the map 
		\begin{equation}\label{eq:zetaAQz} 
		z\mapsto \zeta(A,Q)(z):={\rm TR}(AQ^{-z})
		\end{equation} 
		called the {\it  $\zeta$-regularised trace of $A$} with respect to the weight $Q$, is holomorphic on a half plane  $\Re(z)>\frac{n+a}{q}$,   meromorphic  on the whole complex plane with poles at $d_j=\frac{a+n-j}{q},\, j\in \Z_{\geq 0}$ and the complex residue at this pole  can be expressed as a Wodzicki residue
		$${\rm Res}_{z=d_j} \zeta(A,Q)(z)=\frac{1}{q}{\rm Res}(A\, Q^{-d_j}).$$  
		If  $A $ differs from a differential operator by a trace-class pseudodifferential operator $T$, the $n$-form  \[
		{\rm res}_x\left(A \, \log Q\right)\, dx :=\int_{\vert \xi\vert_x=1} \sigma_{-n}\left(A\, \log Q\right)(x,\xi)\, \dbar_S\xi\, dx .\]
		defines a global density on $M$ which integrates to the 
		extended
		Wodzicki residue $${ \rm
			Res} (A \,\log Q ):=\int_M {\rm res}_x\left(A \,\log Q\right)\, dx$$ of  $   A\, \log Q$.  \\   The zeta function $\zeta  (A,Q)$ is holomorphic at zero and we have
		\begin{equation}\label{eq:PSclassicalopzeta}
		\zeta(A,Q)( 0 )=\lim_{z\to 0}\zeta(A,Q)(z)=  {\rm Tr}(T)
		-  \frac{1}{q}{\rm Res}(A\, \log Q).
		\end{equation} 
	\end{thm}
	
	When $A$ is a differential operator, its   $Q$-weighted $\zeta$-regularised trace $\zeta(A,Q)( 0 )$ is therefore proportional to the extended Wodzicki residue ${\rm Res}(A\, \log Q)$. So in that case, the $\zeta$-regularised trace $\zeta(A,Q)( 0 )$ is local since the extended Wodzicki residue is local in so far as it is expressed as  an integral over $M$ of  the pointwise extended residue (Proposition \ref{prop:reslogxlocalexpression}),  which only depends on finitely many (here one) homogeneous components of the symbol of $A$.
	\\
	In particular, for $A=I$, we have as announced previously, that the {\bf logarithmic  residue} ${\rm Res}(\log Q)$ is well defined, moreover   the zeta function of $Q$ at zero is local    $$\zeta_Q(0):=\zeta(I,Q)(0)=-\frac{1}{q}{\rm Res}(\log Q).$$ 
	
	\section{Heat-kernel expansions revisited}
	\subsection{The heat-kernel expansion in terms of Wodzicki residues}
	
	We  recall  the definition and some properties of the Mellin transform following \cite{FGD} (see also \cite{Je}).
	
	\begin{defn}
		Let $f(t)$ be a locally Lebesgue integrable function over $]0,+\infty[$. The Mellin transform of $f(t)$ is defined as
		\[
		{\mathcal M}(f)(z):=\int_0^\infty f(t)t^{z-1}\,dt.
		\]
		The largest open strip $a<\Re(z)<b$ in which the integral converges is called the fundamental strip.
	\end{defn}
	We recall the following well-known  Inverse Mellin Mapping Theorem \cite[Theorem 4]{FGD}. 
	
	\begin{prop}\label{prop:InverseMellin}
		Let $f(t)$ be continuous in $]0,+\infty[$ with Mellin transform $\phi(z)$ having a fundamental strip $a<\Re(z)<b$.
		\begin{enumerate}
			\item Provided
			\begin{enumerate}
				\item  $\phi(z)$ admits a meromorphic continuation to the strip $(\gamma,b)$ for some $\gamma<a$ with a finite number of poles in the strip, and is analytic on $\Re(z)=\gamma$,
				\item there exists a real number $c$ in $(a,b)$ such that for some $r>1$
				$$ \phi(z)=O(|z|^{-r}),\  \text{ when } |z|\rightarrow\infty \text{ in } \gamma\leq \Re(z)\leq c,$$
				\item $\phi$ admits the singular expansion for $z\in (\gamma,a)$
				$$\phi(z)\cong \sum\limits_{i\geq 0}\sum\limits_{j\geq 0}a_{i,j}\frac{(-1)^{k_j}k_j!}{(z-d_j)^{k_i+1}},$$
			\end{enumerate}
			then $f$ {\bf admits an asymptotic expansion}  at $0$ given by
			$$f(t)=\sum\limits_{i\geq 0}\sum\limits_{j\geq 0} a_{i,j} t^{-d_j}\log^{k_i} t+O(t^{-\gamma}).$$
			\item Provided
			\begin{enumerate}
				\item   $\phi(z)$ admits a meromorphic continuation to the strip $(a,\gamma)$ for some $\gamma>b$ with a finite number of poles in the strip, and is analytic on $\Re(z)=\gamma$,
				\item there exists a real number $c$ in $(a,b)$ such that for some $r>1$
				$$ \phi(z)=O(|z|^{-r}),\  \text{ when } |z|\rightarrow\infty \text{ in } c\leq \Re(z)\leq \gamma,$$
				\item $\phi$ admits the singular expansion for $z\in (c,\gamma)$,
				$$\phi(z)\cong \sum\limits_{i\geq 0}\sum\limits_{j\geq 0}a_{i,j}\frac{(-1)^{k_i}k_i!}{(z-d_j)^{k_i+1}},$$
			\end{enumerate}
			then $f$  {\bf admits an asymptotic expansion}  at $\infty$ given by
			$$f(t)=\sum\limits_{i\geq 0}\sum\limits_{j\geq 0} a_{i,j} t^{-d_j}\log^{k_i} t+O(t^{-\gamma}).$$
		\end{enumerate}
	\end{prop}
	
	\begin{rk}
		In particular, it follows from 2.c) that if $\phi$ is analytic  in some half plane $\Re(z)>c$ then $f(t)= O(t^{-\gamma})$ for any $\gamma>c$ and $f$ is a Schwartz function.
	\end{rk}Here is a useful example to keep in mind for what follows.
	\begin{ex} For any $\lambda>0$ the map  $z\longmapsto \phi(z)=\Gamma(z)\, \lambda^{-z}$ satisfies  the above assumptions as a result of the properties of the Gamma function.  Indeed, on the one hand it is meromorphic  on the complex plane with  simple poles in $\Z_{\leq 0}$. On the other hand, it follows from Stirling's formula (see  e.g. \cite[Proposition IV:1.14]{FB}) which expresses the Gamma function as $\Gamma(z)=\sqrt{2\pi}\, z^{z-\frac{1}{2}}\,e^{-z}\, e^{H(z)}$ for some function $H$  given by a series, that 
		for any $0<\gamma<\delta$ there is   a positive constant  $C_{\gamma, \delta}$ such that  for any  $$\gamma\leq  \Re(z) \leq \delta\Longrightarrow   \vert\Gamma(z)\vert \leq C_{\gamma, \delta} \,   \vert \Im(z) \vert^{\delta-\frac{1}{2}} \,e^{-\gamma}. $$      Since $\lambda^{-z}$ is bounded from above by $\lambda^{-\gamma}$ this shows that the Condition 1. b)  is satisfied for any   $r>1$. The inverse Mellin transform is $f(t)=e^{-t\lambda}$ which defines a Schwartz function on $\R_+$.
	\end{ex}
	Combining these properties of the Gamma function with the results of  Theorem \ref{thm:KVPSop} yields the following useful properties.
	\begin{cor}
		\label{cor:propertiesphiz}
		Let $A(z)\in \Psi(M,E)$ be a holomorphic family    of affine order $\alpha(z)=a-qz$ for some positive real number $q$;  we set
		$$\phi(z):=\Gamma(z)\, {\rm TR}\left(A (z)\right)\quad \alpha:={\rm Max} \left\{0;  \frac{a+n}{q}\right\}.$$
		\begin{enumerate}
			\item   The map $\phi$ is holomorphic on the half-plane $\Re(z)>\alpha$,   it is analytic on any imaginary line $\Re(z)=\gamma$ for $\gamma>\alpha$, and  $\left(\alpha,\infty\right)$  is the largest open strip where $\phi$ is defined. 
			\item $\phi$ admits a meromorphic extension to the whole complex plane with countably many simple poles   $\{d_j:=\frac{a+n -j}{q}, j\in \Z_{\geq 0}\}\cup \Z_{\leq 0}$
			\item  whenever   $z\mapsto  {\rm TR}\left(A (z)\right) $  is uniformly bounded    in   closed strips $\frac{a+n}{q}<\gamma\leq \Re(z)\leq \delta$,  then for some $r>1$ we have the following asymptotic behaviour of $\phi$ at infinity along imaginary lines
			$$ \phi(z)=O(|z|^{-r}),\  \text{ when } |z|\rightarrow\infty \text{ in }\frac{a+n}{q}< \gamma\leq \Re(z)\leq \delta,$$as  a consequence of the corresponding property of the Gamma function described above. 
			\item $\phi$ admits a singular expansion   on any horizontally bounded strip $$\phi(z)\cong \sum\limits_{j=0} \frac{a_j}{z-d_j},$$ as a consequence of the second item.
		\end{enumerate}
	\end{cor}
	
	Let as before, $M$ be an $n$-dimensional closed manifold and $E$ be a finite rank vector bundle over $M$.  The subsequent theorem   follows from Proposition \ref{prop:InverseMellin} applied to $\phi(z)= \Gamma(z)\, {\rm TR}\left(A (z)\right) $.
	\begin{thm}
		\label{thm:mainthm} 
		Let  ${\mathcal A}\colon z\mapsto A(z)\in \Psi(M,E)$ be some   holomorphic family  of operators   of affine order $\alpha(z)=a-qz$ with $q$  some positive real number corresponding to the Mellin transform of some  analytic family $\widetilde A(t), t>0 $  of trace-class operators in $ \Psi(M,E)$. Whenever   $z\mapsto  {\rm TR}\left(A (z)\right) $  is uniformly bounded    on   closed strips $\frac{a+n}{q}<\gamma\leq \Re(z)\leq \delta$, then $f_{\mathcal A}(t)= {\rm Tr}\left( \widetilde A(t)\right)$ admits an asymptotic expansion    at $0$ given by 
				\begin{equation}\label{eq:fAt} f_{\mathcal A}(t)\sim_0  \,\frac{1}{q}\sum\limits_{j\geq 0} a_j t^{\frac{-a-n +j}{q}},\end{equation} with
		\begin{equation}\label{eq:aj}
		a_j= -\frac{1}{q} {\rm Res}\left( A(d_j)\right) \quad\text{for}\quad j>a+n.
		\end{equation} 
		If $A(0)$ differs from a differential operator by a trace-class pseudodifferential operator $T$, the constant term in the expansion (\ref{eq:fAt}) which coincides with the constant term in the Laurent expansion of ${\rm TR}(A(z))$ reads
		\begin{equation}\label{eq:aj1}
		a_j= {\rm Tr}(T)-\frac{1}{q} {\rm Res}\left(A^\prime(0)\right), \quad\text{for}\quad  j=a+n .
		\end{equation} 
	\end{thm}
	Let $Q\in \Psi(M,E)$ be a weight of positive order $q$. We further assume that it is  ``close" to a positive operator i.e.,  its spectrum is concentrated in a cone centered around the positive real line; in particular, it has  spectral cut $\beta=\pi$. 
	
	The holomorphic family ${\mathcal A}: z\mapsto A(z)= A\, Q^{-z}$ is the Mellin transform of the analytic family  $A\,e^{- tQ}, t>0$. Since the map $z\mapsto  {\rm TR}\left(A (z)\right) $  is uniformly bounded    on   closed strips $\frac{a+n}{q}<\gamma\leq \Re(z)\leq \delta$ we can apply Theorem \ref{thm:mainthm}, which yields the following corollary.
	\begin{cor}\label{cor:invMellin}
		The inverse Mellin transform
		$$f_{\mathcal A}(t):={\rm Tr}\left(A e^{-tQ }\right)$$
		of $\phi_{\mathcal A}(z):=\Gamma(z) {\rm TR}\left(A Q^{-z}\right)$ is a Schwartz function on $]0,+\infty[$, which  admits an asymptotic expansion  at $0$ given by
		\begin{equation}\label{eq:ft}
		f_{\mathcal A}(t)\sim_0 \,t^{-\frac{a+n}{q}} \sum \limits_{j\geq 0} a_{j} t^{ \frac{ j}{q}},
		\end{equation}  where
		\begin{eqnarray}\label{eq:aj2}
			a_j &=& -\frac{1}{q} {\rm Res}\left( AQ^{\frac{-a-n+j}{q}}\right) \quad\text{for}\quad j<a+n\nonumber\\
			{\rm and }\quad  a_{a+n}&=&\zeta(A,Q)(0)= {\rm Tr}(T)-\frac{1}{q}{\rm Res}\left(A\, \log Q\right) 
			\end{eqnarray} 
		if 
		  $A\in \Psi(M,E)$ differs from a differential operator by a trace-class operator $T$.
	\end{cor}
	Let  $\Delta\in \Theta(M,E)$  be  an admissible operator and 
		let $E$ be equipped with a hermitian metric, which combined with a Riemannian metric on $M$ induces a (weak) inner product 
		on $\Ci(M,E)$.  We previously saw that the operator $ Q:= \Delta +\pi_\Delta$   defines a weight.
			
	{\bf Notation convention:} 
	Integrating over $M$ the pointwise residue in (\ref{eq:resxlogQ}) gives rise to    the {\bf   logarithmic residue of $\Delta$} defined as
	\begin{eqnarray}\label{eq:reslogARx}
	{\rm Res}(\log \Delta):={\rm Res}(\log (\Delta+\pi_\Delta)). 
	\end{eqnarray} Similarly,  for any $x\in M$ and $\alpha\in \R$ we set
		$${\rm res}_x(\Delta^\alpha):= {\rm res}_x\left(\left(\Delta+\pi_\Delta\right)^\alpha\right);\quad \Res(\phi\,\Delta^\alpha):= \Res\left(\phi\, \left( \Delta+\pi_\Delta\right)^\alpha \right).$$
	We now specialise to the multiplication operator $A=\phi\in C^\infty(M)$ and apply Theorem \ref{thm:mainthm} to   the holomorphic family $A(z)= \phi\, ( \Delta+\pi_\Delta)^{-z}$. 
	\begin{thm}\label{thm:an}  
		For any smooth function $\phi\in \Ci(M)$ we have
		\begin{eqnarray}
		\label{eq:HKresphi}   
		&& {\rm Tr}\left(\phi\, e^{-t\,\Delta}\right) \sim_{t\to 0}\\
		&&\quad\quad-\frac{(4\pi )^{\frac{n}{2}}}{2 }\,\Bigg[{\rm Res} \left(\phi\,\log \Delta\right) \, \delta_{\frac{n}{2}-\left[\frac{n}{2}\right]} \nonumber \\
		&&\quad\quad+ \sum_{k\in\left[0,\frac{n}{2}\right[\cap \Z} \Gamma\left( \frac{n}{2}-k\right)  \,{\rm Res} \left(\phi\,  \Delta^{k-\frac{n}{2}}\right)\, t^{k-\frac{n}{2}}\Bigg]\nonumber.
		\end{eqnarray}
		The local heat-kernel trace  $K_t( \Delta )(x,x)$ of the operator $\Delta$ at the point $x$ is defined by
		$${\rm Tr}\left(\phi\, e^{-t\,\Delta}\right)=\int_M\phi(x)K_t\left(\Delta\right)(x,x)\, \sqrt{{\rm det} g}(x)\, dx\quad\forall \phi\in \Ci(M)$$
		and therefore it has the following asymptotic expansion  
		\begin{eqnarray}
		\label{eq:HKres}   
		&& K_t(\Delta)(x,x) \sim_{t\to 0}\\
		&&-\frac{(4\pi )^{\frac{n}{2}}}{2\, \sqrt{{\rm det} g}(x) }\,\Bigg[{\rm res}_x\left(\log \Delta\right) \, \delta_{\frac{n}{2}-\left[\frac{n}{2}\right]} \nonumber \\
&&		+ \sum_{k\in\left[0,\frac{n}{2}\right[\cap \Z} \Gamma\left( \frac{n}{2}-k\right)  \,{\rm res}_x \left( \Delta^{k-\frac{n}{2}}\right)\, t^{k-\frac{n}{2}}\Bigg]\nonumber.
		\end{eqnarray}
	\end{thm}
	
	\begin{rk}
		Formula (\ref{eq:HKresphi}) compares with known formulae for the spectral action (take $f(\lambda)= e^{-t\lambda}$ in \cite[Formula 1.5]{CC}), in which the non-constant coefficients arise as Dixmier traces. 
	\end{rk} 
	\begin{proof}  It follows from  (\ref{eq:PSclassicalopzeta}) that the constant term in the heat-kernel expansion reads
		$${\rm fp}_{t=0} {\rm Tr}\left(\phi\, e^{-t\left(\Delta+\pi_\Delta \right)}\right)=\zeta\left(\phi, \Delta+\pi_\Delta\right)(0)= \frac{1}{2}\,{\rm Res} \left(\phi\,\log \Delta\right).$$ 
		Formula (\ref{eq:HKresphi}) then follows from Corollary \ref{cor:invMellin} with $q=2$ (cf. footnote in the introduction). 
		
		Since this holds for any smooth function $\phi$, formula (\ref{eq:HKres}) follows.
	\end{proof}
	
	\begin{rk}
		Combining (\ref{eq:reslogARx}) with (\ref{eq:resxlogQ})  applied to $\Delta$   
		$$
		{\rm res}_x(\log \Delta)=  \frac{1}{2\pi i } \int_{\vert \xi\vert_x=1}\, \int_\Gamma \log \lambda\, \sigma_{-q-n}( \Delta+\pi_\Delta-\lambda)^{-1 }(x,\xi)\, d\lambda \, \dbar_S\xi,  
		$$
		yields
		$$ a_{\frac{n}{2}}(x)= -\frac{  1}{4i \,\pi^{\frac{n}{2}+1} \, \sqrt{{\rm det} g}(x) } \int_{\vert \xi\vert_x=1}\, \int_\Gamma \log \lambda\, \sigma_{-q-n}( \Delta+\pi_\Delta-\lambda)^{-1 }(x,\xi)\, d\lambda \, d_S\xi.$$ 
		This compares with  similar formulae in the literature as for example \cite[(2.11)]{BFKM} 
		$$ a_{\frac{n}{2}}(x)= \frac{  1}{ 2} \int_{\vert \xi\vert_x=1}\, \int_0^\infty  \sigma_{-q-n}( \Delta+\pi_\Delta-\lambda)^{-1 }(x,\xi)\, d\lambda \, \dbar_S\xi.$$  
	\end{rk}
	
	\subsection[The case of geometric operators]{The case of geometric operators}
	\label{ssec:geom}
	We now single out a class of differential operators we call {\it geometric differential operators} (also considered in \cite{MP}) i.e.,  differential operators $A=\sum_{\vert \alpha\vert \leq d}a_\alpha(x)D_x^\alpha\in \Theta(M,E)$  where $d$ is the order of the operator and whose coefficients $a_\alpha(x)$  are given by an algebraic expression in terms of the jets at the point $x$ of the metric  on $M$ and a connection on $E$. 
	
	The Laplace-Beltrami operator $ \Delta_g$ on a closed Riemannian manifold $(M,g)$ is a geometric differential operator of order $2$. Another example is the Bochner-Laplacian   $$\Delta^\nabla:=  {\rm Tr}\left(\nabla^{T^*M\otimes E}\circ \nabla^E  \right)\in \Psi(M,E)$$   built from a connection $\nabla^E$ on $E$ and the induced connection  $\nabla^{T^*M\otimes E}$ on the tensor product $T^*M\otimes E$. Here the trace is taken over the two factors of $T^*M$. The square of a Dirac operator on a  spin manifold $(M,g)$, which differs from the corresponding Bochner-Laplacian by a term proportional to the scalar curvature, is geometric. More generally, the square $A=D^2$ of  a Dirac-type operator $D=\sum_{i=1}^n c(e_i)\nabla^E_i$ is a geometric operator; here $\nabla^E$ is a Clifford connection on a Clifford bundle $E$ over $M$, $c$ the Clifford multiplication and $e_i,i=1,\cdots, n$ an orthonormal frame of the cotangent bundle at a point
	$x$. All these examples fall in the class of  Laplace-type differential operators considered previously. 
	\begin{prop} \label{prop:resAlogQgeom}
		Let $A$ and $\Delta$ be two geometric differential operators in  $\Theta(M,E) $ and let $\Delta$ be admissible.   The extended  local residue ${\rm res}_x(A\, \log \Delta)$ at a point $x$  and   for any real number $\alpha$, the local residue ${\rm res}_x \left( A  \, \Delta^\alpha\right)$,  
		are algebraic expressions of the jets of the metric and the connection. 
	\end{prop}
	
	\begin{proof} 
		That ${\rm res}_x(A\, \log \Delta)$ is an algebraic expression in the $x$-jets of a finite number of homogeneous components of the symbols follows from Proposition \ref{prop:reslogxlocalexpression}, which tells us that the extended local residue is an algebraic expression in the $x$-jets of a finite number of homogeneous components of the symbols of $A$ and $\Delta$ combined with the fact that these  homogeneous components are themselves algebraic expressions in the jets of the metric and the connection.\\
		That ${\rm res}_x(A\,  \Delta^\alpha)$ is an algebraic expression in the $x$-jets of a finite number of homogeneous components of the symbols  can be shown similarly.
	\end{proof}
	
	The case of the Laplace-Beltrami  operator $\Delta_g$    on an   $n$-dimensional Riemannian manifold $(M,g)$ is particularly relevant for us since it enables to capture the scalar curvature  as a Wodzicki (possibly extended, depending on the dimension) residue.  
	Indeed,  the scalar curvature $\mathfrak s_g$ is proportional to the coefficient  $b_1(x)$  in the heat-expansion  
	\begin{eqnarray}\label{eq:Deltat}
	 	K_t \left(\Delta_g \right)(x,x)&\sim_{t\to 0}& t^{-n\frac{ n}{2} } \sum \limits_{j\geq 0} a_{k}(x) \ t^{\frac{k}{2}}\nonumber \\ &=& t^{- \frac{ n}{2}}  \sum \limits_{j\geq 0} b_{j}(x)\, t^j,
		\end{eqnarray}
	 where we have set $b_j=a_{2j}$	since the coefficients $a_{2j+1}$ vanish \cite{Gi}, \cite{Ro}. \\
We have \cite[Formula (5a)]{MS} 
	$$\mathfrak  s_g=3\, b_1,$$
	independently of the dimension $n$.\\
	This combined with   Theorem  \ref{thm:an} leads to the following expressions of the scalar curvature in terms of Wodzicki residues.
	 	\begin{prop} 
		When $n=2$ we have
		\begin{equation}\label{eq:sg2} 
		\mathfrak s_g(x)= -\frac{ 6\pi  }{ \sqrt{{\rm det} g}(x)}\, {\rm res}_x\left(\log \Delta\right). 
		\end{equation} 
		If $n>2$ then,
		\begin{equation}\label{eq:sgn}  
		\mathfrak s_g(x)= -\frac{3}{2}\,  (4\pi )^{\frac{n}{2}}\, \Gamma\left( \frac{n}{2}-1\right) \,\frac{{\rm res}_x \left(   (\Delta+\pi_\Delta)^{1-\frac{n}{2}}\right)}{\sqrt{{\rm det} g}(x)}. \end{equation} 
		
	\end{prop}
	In dimension $2$, which is the case we are going to focus on in the sequel, we have
	\begin{equation}\label{eq:scalarcurvature}
	\langle \mathfrak s_g , \phi\rangle_g= -  6\pi  \,  {\rm Res} \left(\phi\, \log    \Delta_g\right)\quad \forall \phi \in \Ci(M).
	\end{equation}



\section{Holomorphic families on Hilbert modules} 
 We  carry out to   pseudodifferential operators on bundles of von Neumann Hilbert modules the constructions  of Section \ref{sec:1} on closed manifolds. 

\subsection[Finite type Hilbert modules]{Finite type Hilbert modules}
\label{ssec1}
We start out recalling the setup of finite type Hilbert modules closely following  \cite{BFKM} and \cite {Sch}\footnote{Note that in Schick's paper these objects are called $\mathcal A$-Hilbert spaces to distinguish them from Hilbert $C^*$-modules.}.

Let ${\mathcal A}$ be a von Neumann algebra equipped with a {\bf finite} trace $\tau\colon \mathcal A\to \mathbb C$. 
This means that $\mathcal A$ is a unital $\C$-algebra
with a $\star$ operation, and the following properties are satisfied:
\begin{enumerate}
\item $ (\cdot, \cdot)\colon  {\mathcal A}\times {\mathcal A} \longrightarrow \C$, defined by $( a, b) :=\tau ( ab^\star )$ is a scalar product and
the completion  $ {\mathcal A}_2$ with respect to this scalar product is a separable
Hilbert space.
\item ${\mathcal A}$ is weakly closed  when viewed as a subalgebra of the space ${\mathcal L}\left({\mathcal A}_2 \right)$ of linear, bounded operators on ${\mathcal A}_2$ (identifying  elements of ${\mathcal A}$ with the corresponding left translations in ${\mathcal L}\left({\mathcal A}_2\right)$).

\item  The trace is normal, i.e. for any monotone increasing net, $(a_i)_{i\in I}$ such
that $a_i \geq  0 $ and $a={\rm sup}_{i\in I}a_i$  exists in ${\mathcal A}$, one has $ {\rm tr}_{\tau}(a)= {\rm sup}_{i\in I} {\rm tr}_{\tau}(a_i)$.
\end{enumerate}
A right $\mathcal A$-Hilbert module is a Hilbert space $\mathcal W$ with a continuous right $\mathcal A$-action that admits an $\mathcal A$-linear isometric embedding into $\mathcal A_2\otimes H$ for some Hilbert space $H$. The Hilbert module $\mathcal W$ is called of {\bf finite type} if the space $H$ can be chosen finite dimensional.
We denote by $\mathcal L_{\mathcal A}(\mathcal W)$ the von Neumann algebra of bounded $\mathcal A$-linear operators on $\mathcal W$.
The (unbounded) trace on $\mathcal L_{\mathcal A}(\mathcal W)$ induced by $\tau$ and by the usual trace on $\mathcal L(H)$ is denoted ${\rm tr}_\tau$.
\begin{ex}
Let  $\Gamma$ be a countable group and $\ell^2(\Gamma)$ be the Hilbert space of square integrable complex valued functions on $\Gamma$.  
The von Neumann algebra $\mathcal N\Gamma$ consists by definition of all bounded operators on $\ell^2(\Gamma)$ that commute with the left convolution action of $\Gamma$. It contains $\mathbb C\Gamma$ as a weakly dense subset, and on $\mathbb C\Gamma$ the canonical trace $\tau$ is given by 
$$
\tau\left(\sum a_\gamma \gamma\right)=a_e$$
 where $e$ is the unit element in $\Gamma$.
Then $\mathcal W=\ell^2(\Gamma)$ is a finite type $\mathcal N\Gamma$-module, indeed $\ell^2(\Gamma)\simeq (\mathcal N\Gamma)_2$.
\end{ex}
\begin{ex} In particular, for $\Gamma=\mathbb Z^n$, the Fourier transform gives an isometric $\mathbb Z^n$-equivariant isomorphism $\ell^2(\mathbb Z^n)\to L^2(\T^n)$, where $\T^n$ is the $n$-dimensional torus. 
Therefore $\mathcal N\mathbb Z^n$ coincides with the commutant $\mathcal L(L^2(T^n))^{\mathbb Z_n}$ of the $\mathbb Z_n$-action on $L^2(\T^n)$, and one obtains an isomorphism  $\mathcal N\mathbb Z^n\simeq L^\infty(\T^n)$.  The canonical trace $\tau\colon L^\infty(\T^n)\to \mathbb C$ is
\begin{equation}\label{ex:tauZn}
\tau(f)=\int_{\T^n}fd\mu\ ,
\end{equation}
where $\mu$ is the measure on $\T^n$ induced by the canonical Lebesgue  measure on $\R^n$.
\end{ex}

On a manifold $M$, an {\bf $\mathcal A$-Hilbert module bundle} $\mathcal E\to M$  is a locally trivial bundle with fibre a (finitely generated, projective) $\mathcal A$-Hilbert module $\mathcal W$, the transition functions being    isometries of  $\mathcal A$-Hilbert modules.
\begin{rk}
The space of $L^2$-sections of an $\mathcal A$-Hilbert module bundle $\mathcal E\to M$ is an $\mathcal A$-Hilbert module. In fact, if $U$ is a subset of $ M$ such that $M\setminus U$ has measure zero and $\mathcal E_{|U}\simeq U\times \mathcal W$ (where $\mathcal W$ is the fibre), then $L^2(M,\mathcal E)\simeq L^2(U, \mathcal E_{|U})\simeq L^2(U)\otimes \mathcal W$. The set $U$ can be chosen for example as being the union of the interiors of the top-order cells of a triangulation of $M$. 
\end{rk}
\begin{ex}(Flat $\mathcal N\Gamma$-Hilbert module)
\label{ex:cov} Let $M$ be a closed manifold with $\pi_1(M)=\Gamma$, let $\pi\colon \widetilde M\to M$ be a universal covering of $M$, with $\Gamma$ acting on the right by deck transformations.
 Because the left $\Gamma$-action   and the right $\mathcal N\Gamma$-action   on $\ell^2(\Gamma)$ commute, $\mathcal H:=\widetilde M\times_\Gamma \ell^2 (\Gamma)$ is a finitely generated projective bundle of (right) $\mathcal N\Gamma$-Hilbert modules over $M$. Moreover, $\mathcal H$ is endowed with a flat structure since the transition functions are locally constant.
 \end{ex}

The flat bundle $\mathcal H\to M$ can be used to describe the analysis on the universal covering. In fact, there is a well-known correspondence between $L^2(M, \mathcal H)$ and $L^2(\widetilde M)$, which translates twisted differential operators on one hand with $\Gamma$-invariant differential operators on the other. We refer for example to \cite[7.5]{Sch} or \cite[Prop. E.6]{PzSc} for the complete dictionary for spaces,  operators, and $L^2$-invariants.
 Let us illustrate this in the case of the $n$-torus discussed above. 
\begin{ex}
\label{ex:covTn}
To the universal covering $\pi\colon \R^n \to \T^n$ with fundamental group $\Z^n$ corresponds the finitely generated projective bundle $\mathcal H:=\R^n\times_{\Z^n} \ell^2 (\Z^n)$   of (right) $L^\infty(\T^n)$-Hilbert modules over $\T^n$.  In this correspondence $L^2$-functions on $\R^n$ are viewed as  $L^2$-sections of the bundle ${\mathcal H}$ over $\T^n$ via the map which sends $f$ to $\hat f\colon x \mapsto \sum_{\gamma \in \Z^n} f(\gamma \pi^{-1}(x))\otimes [\pi^{-1}(x), \gamma]$.  This induces an isometry 
\begin{equation}\label{PHI}
\Phi\colon  L^2(\R^n)\simeq L^2(\T^n)\otimes \ell^2(\Z^n)\longrightarrow L^2(\T^n,   \mathcal H)
\end{equation}
which sends $\{f\in \Ci(\R^n) : \sum_{\gamma \in \Z^n} |f(\gamma x)|^2 <\infty \; \forall x\in \R^n \}$ to $\Ci(\T^n,\mathcal H)$.
 \end{ex}

\subsection[Pseudodifferential operators on Hilbert modules]{Pseudodifferential operators on bundles of Hilbert modules}
\label{ssec2}    

 We  describe    pseudodifferential operators on bundles of Hilbert modules, following  \cite[2.2-4]{BFKM}. 

 Let $\left(M,g\right)$ be a closed Riemannian manifold of dimension $n$ and let  $p\colon {\mathcal E}\to M$ be a bundle of finitely generated projective ${\mathcal A}$-Hilbert modules with fibre ${\mathcal W}$.
      \\ A  linear operator $A\colon  C^\infty(M,\mathcal E)\to  C^\infty(M, \mathcal E)$ is a (classical) pseudodifferential ${\mathcal A}$-operator of  order $a\in \C$ if in some atlas of ${\mathcal E}  \to M$ it is of the form $A=\sum_{j=1}^J A_j+R$ where 
      \begin{itemize}
      \item $R\colon C^\infty(M,\mathcal E)\to  C^\infty(M,\mathcal E)$ is a smoothing operator, namely with Schwartz kernel  given by a smooth  section of the bundle ${\mathcal L}_\mathcal A \to M\times M$ of ${\mathcal A}$-linear bounded operators whose fibre at $(x,y)\in M\times M$ is the Banach space of ${\mathcal A} $-linear operators from the fibre  ${\mathcal E}_y$ to the fibre ${\mathcal E}_x$, 
      \item the operators $A_j$ are  properly supported operators (meaning that the canonical projections $M\times M\to M$ restricted to the support of the Schwartz kernel are proper maps) from $C^\infty(M,\mathcal E)$ into itself. This implies in particular that the $A_j$ send $C^\infty_0(M,\mathcal E)$ into itself.
         In any coordinate chart $A_j$ is therefore of the form given in  \eqref{eq: class.op}, for some  symbol  $\sigma\in C^\infty \left(U
     \times \R^n, {\rm End}_{\mathcal A}\mathcal E\right)$  which is  asymptotically polyhomogeneous at infinity.
            \end{itemize} 
            For $a\in \C$, let  $\Psi^a(M, {\mathcal E})$ be the set of classical pseudodifferential operators of order $a$; then $\Psi^{-\infty}(M, {\mathcal E}):=\cap_a\Psi^a(M, {\mathcal E})$  corresponds to the smoothing operators, and $\Psi\left(M, {\mathcal E}\right):=\cup_a\Psi^a(M, {\mathcal E})$. 
 The subalgebra of differential operators will be denoted $\Theta(M, \mathcal E)$.
         \begin{rk}  \label{rk:Atorus} Going back to the example of the $n$-torus $\T^n$ and its $\Z^n$-covering $\pi:\R^n\to\T^n$, we  know that $\Z^n$-invariant  (i.e., $A \circ t_a^*= t_a^*\circ A$ for any $a\in \Z^n$ with $t_a(x)=x+a$)  differential operators on $\R^n$  give rise to elements of $\Theta(M,{\mathcal H})$. This raises the question  how  the algebra $\Psi(\T^n,\mathcal H)$  relates to the algebra of $\Z^n$-invariant classical pseudodifferential operators on $\R^n$ studied by Ruzhanski and Turunen in \cite{RT}
or equivalently the $\theta=0$  instance of the algebra $\Psi(\T^n_\theta)$    described in Section 4. 
  \end{rk}
  For any $A\in \Psi^a(M, \mathcal E)$, $a\in \R$, $A $ defines a bounded linear map on the $H^s$-Sobolev closure $H^s\left(M, \mathcal E\right)$ of $C^\infty\left(M, \mathcal E\right)$ with values in $H^{s-a}\left(M, \mathcal E\right)$, for any $s\in \R$. We refer the reader to \cite[\S 2.2]{BFKM} for Sobolev closures and to \cite[Prop. 2.7]{BFKM} for the property in the case $s=a$.

The notions of admissibility (Definition \ref{def:admiss1}) and Agmon angle (Lemma \ref{lem:agmon angle})   carry out from the closed manifold case to  the setup of pseudodifferential operators on bundles of von Neumann Hilbert modules, taking into account that here the spectrum is not necessarily purely discrete.
    \begin{defn}\label{def:agmonL2}
   Let $A$ be an operator  in $ \Psi^a(M, {\mathcal E})$. For an angle $\beta$ and for $\epsilon >0$, denote $V_{\beta,\epsilon}:=\{z\in \C\;: |z|<\epsilon \}\cup \{z\in \C\setminus 0\,:\; {\rm arg} z \in (\beta-\epsilon, \beta+\epsilon)\}$. 
    Then $\beta$ is called an {\bf Agmon angle} for $A$ if there is some $ \epsilon >0$ such that $\spec( A)\cap V_{\beta, \epsilon}=\emptyset$, where $\spec (A)$ stands for the spectrum of $A$.
                     \end{defn}
 In particular an operator with Agmon angle is elliptic and invertible.

\begin{defn} We  call {\bf weight}   an operator in $\Psi(M, {\mathcal E})$ of positive order   that admits an Agmon angle.
We call {\bf admissible} an operator $A\in \Psi(M, {\mathcal E})$ which is a weight modulo a smoothing perturbation, i.e. if it has positive order, and there exists a smoothing operator $R$ such that $A+R$ has an Agmon angle.
\end{defn}
The constructions of complex powers and the logarithm recalled in Section \ref{sseclog} extend word for word to the setting of Hilbert modules. 

Let $A$ be a weight in $\Psi^a(M,\mathcal E)$ with spectral cut $\beta$. 
   Then for  $\Re (z)>0$, its complex powers (see  \cite[(2.8)]{BFKM} for further details, and compare with \eqref{eq.cxp})
    \begin{equation}\label{eq.cxpL2}
    A_\beta^z= \frac{i}{2\pi} \int_{\Gamma_\beta} \lambda^z \, (\lambda-A)^{-1}\, d\lambda, \end{equation} 
   and respectively the operators (as in \eqref{eq.Log})
     \begin{equation}\label{eq.LogL2}L_\beta(A, z)= \frac{i}{2\pi} \int_{\Gamma_\beta} \log_\beta \lambda \;\;\lambda^z \, (\lambda-A)^{-1}\, d\lambda, 
     \end{equation} 
  are well defined bounded linear maps from $H^s\left(M, \mathcal E\right)$, $s\in \R$,  
  with values in $H^{ s-a\Re(z) }\left(M, \mathcal E\right)$ respectively  in $H^{ s-a\Re(z)  + \epsilon}\left(M,  \mathcal E\right)$, for any $\epsilon >0$. 
  Here $\Gamma_\beta$ is  a closed contour  in $\C\setminus \{re^{i\beta },\; r\geq 0\} $ around  the spectrum of $A$ oriented clockwise. These definitions extend   to the whole plane inductively on $k\in \N$  setting 
  $$A_\beta^z:=A^k \;A_\beta^{z-k} \;\;\;, \;\text{respectively} \;\;L_\beta(A, z):=A^k\;  L_\beta(A, z-k), \text{ for }\;\Re(z)<k.
  $$ 
 $A_\beta^z$ is an operator in $\Psi(M, \mathcal E) $ of order $az$ for any complex number $z$, and the logarithm 
 $$\log_\beta (A):=L_\beta(A, 0)
 $$ of $A$ is a bounded linear map from  $H^s\left(M, \mathcal E\right)$ to $H^{s+\epsilon}\left(M, \mathcal E\right)$, $\forall\epsilon >0$.
One has by construction  $ \log_\beta A\,  A_\beta^{z}=
        A_\beta^{z}\, \log_\beta A$, $\forall z\in \C$.
   \begin{rk} 
       \label{rk:beta2}Using the same convention as in Remark \ref{rk:beta}, we usually drop
        the explicit mention of the dependence of $\beta$. 
        \end{rk}
        In a local trivialisation, the symbol of $\log_\beta A$ reads (see e.g.  \cite{Sc2})
        \begin{equation}\label{eq:symbollogL2}
        \sigma(\log_\beta A)( x,\xi)=a\, \log \vert \xi\vert I +
        \sigma_{\rm cl}(\log_\beta A)(x, \xi)
        \end{equation}
         where $a$ denotes the order of $A$ and $\sigma_{\rm cl}(\log A)$ is a classical symbol of
        order zero with  homogeneous components   $\sigma_{-j}(\log  A)$ of degree $-j, j\in \Z_{\geq 0}$.

\subsection[The extended $\tau$-residue]{The extended $\tau$-Wodzicki residue}
 Mimicking the definition of the pointwise residue \eqref{eq:resx} on closed manifolds, 
          for any $A\in\Psi\left(M, {\mathcal E}\right)$ with symbol $\sigma(A)$ in a local chart of $M$ at a point $x$, we  call
             \begin{equation}\label{eq:restaux}\quad {\rm res}_x^\tau (A):=  \int_{|\xi|_x=1}{\rm tr}^\tau  (\sigma_{-n}(A))(x,\xi)\, \dbar_S\xi,
                             \end{equation} the pointwise ${\tau}$-Wodzicki residue of $A$.
              As in the closed manifold case, one proves that $ {\rm res}_x^{\tau}(\sigma(A))\, dx$  defines a global  density on $M$ \cite{Wo}, so we can define
              the ${\tau}$-Wodzicki residue of $A$  (compare with (\ref{eq:resA}))
                        \begin{equation}
                                \label{eq:resAtau}{\rm Res}^\tau(A) =\int_M {\rm res}_{x }^\tau(A)\, dx=   \int_M dx\int_{\vert \xi\vert_x=1} {\rm tr}^\tau\left(\sigma_{-n}(A)(x,\xi)\right)\,  \dbar_S\xi.
\end{equation}

\begin{rk} Definition \eqref{eq:restaux} relates to the definitions of the Wodzicki residue by Benameur--Fack  \cite[Def. 9]{BF}  and Vassout \cite{Va} in the context of measured foliations.
\end{rk}
              The {\bf pointwise Wodzicki $\tau$-residue}  extends to the logarithm $\log A$ of an admissible invertible operator $A\in \Psi(M, {\mathcal E})$. At a point $x$ we set
                         \begin{equation}
                         \label{eq:resxL2}\quad {\rm res}_x^{\tau}(\log A):=  \int_{|\xi|_x=1}
                         {\rm tr}^{\tau} (\sigma_{-n}(\log A))(x,\xi)\, \dbar_S\xi,
                         \end{equation}      
                  Mimicking the definition (\ref{eq:reslogARX}), for any admissible operator $A\in \Psi(M,{\mathcal E})$, we set
                  \begin{eqnarray}\label{eq:reslogAR}
                  \res_{x}^\tau(\log A):=\res_{x}^\tau(\log (A+R))\quad \forall x\in M,
                   \end{eqnarray}
                    where $R$ is any smoothing operator such that $A+R$ has an Agmon angle. \\  In the subsequent paragraph, we show as in the closed manifold case, that $ \res_x^\tau (\log A)\, dx$ defines a global density. 
                \subsection[The $\tau$-Wodzicki residue as a complex residue]{The $\tau$-Wodzicki residue as a complex residue}
                    For holomorphic families of operators in $\Psi(M, {\mathcal E})$, whose explicit definition we omit here since they are  defined in the same manner as in the closed case (see Definition \ref{defn:holop}), we give the Hilbert-module counterpart of Theorem \ref{thm:KVPSop}. For this we need as in the case of operators on closed manifolds, the cut-off integral $\displaystyle\cutoffint_{\R^n} {\rm tr}_x^\tau\sigma (x, \xi) \, \dbar\xi$ of a local  symbol $ \sigma$ of a classical operator, which is defined in the same way as in (\ref{eq:cutoff}), only replacing the fibrewise trace by ${\rm tr}^\tau$.  
                   \begin{prop}\label{prop:KVPSopL2}  For any holomorphic family  $ A(z)\in \Psi(M, {\mathcal E})$ of classical operators parametrised by $   \C$, with local symbols $\sigma(z)$ and holomorphic order $   -qz+a$ for some positive $q$ and some real number $a$,
                      \begin{enumerate}
                    \item the meromorphic map $z\mapsto {\rm TR}_x^\tau(A(z)):=\cutoffint_{\R^n} {\rm tr}^\tau\left(\sigma(z)( x, \xi)\right) \, \dbar\xi$ integrates over $M$ to  the map  $$z\mapsto {\rm TR}^\tau\left(A(z)\right):= \int_M {\rm TR}_x^\tau(A(z))\, dx $$  which is  meromorphic with simple poles $d_j:=\frac{a+n-j}{q}, j\in \Z_{\geq 0}$.
                    \item $\,${\rm\cite{KV}} The complex  residue at the point $d_j$  is given by:
                    \begin{equation}\label{eq:classicalKVtau}{\rm Res}_{z=d_j} {\rm TR^\tau}\left(A(z)\right)= \frac{1}{q}{ \rm
                    Res}^\tau (A(d_j) ).
                    \end{equation} 
                    \item  $\,${\rm\cite{PaSc}}  If $A(d_j)$  lies in $\Theta(M,E)$  i.e., if it is a differential operator, then $A^\prime (d_j)$  which need not be a classical pseudodifferential operator, nevertheless has a well defined 
                   Wodzicki residue $${ \rm
                     Res}^\tau (A^\prime(d_j) ):=\int_M {\rm res}_x^\tau\left(A^\prime (d_j)\right)\, dx$$
                     where
                     \[
                      {\rm res}_x^\tau\left(A^\prime (d_j)\right):=\int_{\vert \xi\vert_x=1} {\rm tr}^\tau\left(\sigma_{-n}\left(A^\prime (d_j)\right)\right)(x,\xi)\, \dbar_S\xi
                     \]
                     at the pole  $d_j$    and we have  
                     \begin{equation}\label{eq:PSclassicaloptau}
                   {\rm fp}_{z=d_j}{\rm TR}^\tau\left(A  (z)\right)=  
                   \frac{1}{q}{ \rm
                    Res}^\tau(A^\prime (d_j)).
                    \end{equation}   \end{enumerate} 
                   \end{prop}
   As in Theorem  \ref{thm:mainthm} we   deduce the asymptotic expansion of the inverse Mellin transform of traces ${\rm TR}^\tau(A(z))$ of holomorphic families $A(z)$. Substituting ${\rm TR}^\tau$  to ${\rm TR}$ in Theorem
                    \ref{thm:mainthm} tells us that, if  $f $ is continuous on $]0,+\infty[$ with Mellin transform  $z\mapsto{\rm TR}^\tau\left(A (z)\right)$ for some holomorphic family  $A(z)\in \Psi(M,{\mathcal E})$     of affine order $\alpha(z)=a-qz$ with $q$  some positive real number, then  $f$ admits an asymptotic expansion    at $0$ given by
                     $$f(t)= \frac{1}{q}\sum\limits_{j\geq 0} a_j t^{-d_j} +O(t^{-\gamma}),$$ with
                     \begin{equation}\label{eq:ajL2}
                     a_j= -\frac{1}{q} {\rm Res}^\tau\left( A(d_j)\right) \quad\text{for}\quad d_j>0,
                     \end{equation} 
                     and constant term \begin{equation}\label{eq:ajbisL2}
                     a_j =-\frac{1}{q} {\rm Res}^\tau\left(A^\prime(0)\right) \quad\text{for}\quad d_j=0.
                     \end{equation} 

Theorem \ref{thm:zetaresM} extends in a straightforward manner.
\begin{thm}\label{thm:mainthmL2}
Given a weight $Q\in \Psi(M,{\mathcal E})$ of order $q\in \R_+$ and any operator   $A\in \Psi(M,{\mathcal E})$ (so not necessarily admissible) of order $a\in \R$,  the map
\begin{equation}\label{eq:zetaAQztau} 
z\mapsto \zeta^\tau(A,Q)(z):={\rm TR}^\tau(AQ^{-z})
\end{equation} 
is holomorphic on the half plane $\Re(z)>\frac{n+a}{q}$ and defines a meromorphic map on $\C$
called the {\it  $\zeta^\tau$-regularised trace of $A$} with respect to the weight $Q$,  with poles at $d_j=\frac{a+n-j}{q},\quad j\in \Z_{\geq 0}$. The complex residue at such a  pole is related to the Wodzicki $\tau$-residue of $A Q^{-d_j}$ by
$${\rm Res}_{z=d_j} \zeta^\tau(A,Q)(z)=\frac{1}{q}{\rm Res}^\tau(A Q^{-d_j}).$$
For any differential operator   $A\in \Theta(M,{\mathcal E})$, the $n$-form ${\rm res}^\tau_x(A\, \log Q)\,dx$ defines a global density on $M$ which integrates to the 
extended
Wodzicki $\tau$-residue of  $   A\, \log Q$. The $\zeta^\tau$-regularised trace $\zeta^\tau(A,Q)(z)$ is holomorphic   at zero and we have
 \begin{equation}\label{eq:zetatauzero}\zeta^\tau(A,Q)(0) =\lim_{z\to 0} \zeta^\tau(A,Q)(z)=-\frac{1}{q}{\rm Res}^\tau(A\, \log Q).
 \end{equation} 
\end{thm}
\begin{rk}
 Let $A, \Delta$ be differential operators in  $\Theta(M,\mathcal E)$  with $ \Delta$  admissible. Then exactly as in Proposition \ref{prop:reslogxlocalexpression} one can prove that the pointwise extended Wodzicki residue   ${\rm res}^\tau_{x}\left(A\,\log \Delta\right)$ is   an algebraic expression  in the  coefficients of $  A$ and in the $ x$-jets  of   the coefficients of  $  \Delta$  at that point.     
\end{rk}

	\subsection{The $\tau$-index as an extended $\tau$-residue}
	
	The above constructions extend to $\Z_2$-graded vector bundles. Let $  \mathcal E=\mathcal E_+\oplus  \mathcal E_- $ be a $\Z_2$-graded bundle of finite type $\mathcal A$-Hilbert modules over $M $ and let $D_\pm\colon C^\infty \left(M, \mathcal E_\pm\right)$ $\longrightarrow C^\infty \left(M,  \mathcal E_\mp\right)$ be two elliptic  differential operators of positive order $d$. We assume that the operators $D_+$ and $D_-$ are formally adjoint to each other which we write $D_-=D_+^*$. 
	Hence 
	\begin{displaymath}
	D:= \begin{bmatrix}
	0 & D_-  \\
	D_+ & 0  
	\end{bmatrix} 
	\end{displaymath}
	is essentially selfadjoint. 
	Let $
	\Delta:=D^2= \begin{bmatrix}
	D_-D_+ & 0  \\
	0 & D_+D_- 
	\end{bmatrix}
	= \Delta_+\oplus \Delta_-$. 
		The ellipticity of $D$ implies that the projection $\pi_\Delta$ onto $\Ker \Delta=\Ker D$ is a smoothing operator of finite $\tau$-rank \cite{BFKM, Sch}.
	Therefore one defines the $\tau$-dimension of the $\mathcal A$-Hilbert modules $\Ker D_\pm$ as
	$${\rm dim}_{\tau}(\Ker D_\pm):= \tau(\pi_{\Delta_\pm})\in \R
$$
and the difference
\begin{equation}
{\rm ind}^\tau D_+:={\rm dim}_{\tau}(\Ker D_+)-{\rm dim}_{\tau}(\Ker D_-)\in \R
\end{equation}
is called the {\bf $\tau$-index} of the operator $D$.  In the case $\mathcal A=\C$ this is the usual definition of the Fredholm index of the operator.
	
	\begin{rk} \label{rk:inv}There exist smoothing perturbations $R$ of the nonnegative selfadjoint operator $D^2$ such that $D^2+R$ is invertible. Since we are in a von Neumann algebraic setting, this follows for example from  \cite[Proposition 2.10]{LP1}, see also \cite{LP2}.
\end{rk}
Therefore there exist operators $R, R'$ such that $D_-D_+ +R$ and  $D_+D_-+ R'$ are elliptic and nonnegative and hence so are their leading symbols nonnegative. 
	Thus $\Delta, \Delta_+, \Delta_- $ define  admissible operators  with spectral cut $\pi$.\\
	Consequently, we can define the {\bf pointwise extended super $\tau$-residue} of $\log \Delta$ as  the difference of the pointwise extended residues of $\log \Delta_+$ and $\log \Delta_-$
	$${\rm sres}^\tau_x\left(\log \Delta \right)(x):={\rm res}^\tau_x\left(\log \left(\Delta_++ R'\right)\right)- {\rm res}^\tau_x\left(\log\left(\Delta_-+ R\right)\right).$$              
	It  can be integrated over $M$ to build the {\bf extended super $\tau$-residue } $${\rm sRes}^\tau\left(\log \Delta \right):= \frac{1}{(2\pi)^n} \int_M {\rm sres}^\tau_x\left(\log \Delta \right)(x)\, dx.
	$$

	\begin{cor}\label{cor:indres}
		The $\tau$-index of $D_+$ is a local expression proportional to the extended Wodzicki (super) $\tau$-residue of the logarithm of $\Delta$
		\begin{equation}\label{eq:indreslog}
		{\rm ind}^\tau (D_+)=  -\frac{1}{2d}\, {\rm sRes}^\tau\left(\log(\Delta)\right).
		\end{equation}
	\end{cor}
	
	\begin{proof} The McKean-Singer formula combined with a Mellin transform yields for any positive real number $t$ 
		and for any complex number $z$ 
		\begin{eqnarray*}
			{\rm ind}^\tau (D_+)
			&=&  {\rm Tr }^\tau\left(e^{- t \left(\Delta_++\pi_{\Delta_+} \right)} \right)-  {\rm Tr }^\tau\left(e^{-t \left(\Delta_-+\pi_{\Delta_-} \right)} \right)\\
			&=& \zeta^\tau_{   \Delta_++\pi_{\Delta_+} }(z)-    \zeta^\tau_{  \Delta_-+\pi_{\Delta_-}}(z)  \\
					&=& -\frac{1}{2d}
					\left({\rm Res}^\tau\left(\log\left(\Delta_++ \pi_{\Delta_+}\right)\right)- {\rm Res}^\tau\left(\log\left(\Delta_-+ \pi_{\Delta_-}\right)\right)\right)\\
					&=& -\frac{1}{2d}\, {\rm sRes}^\tau\left( \log \Delta \right).
				\end{eqnarray*}
			\end{proof}
\subsection[Locally equivalent operators and Atiyah's $L^2$-index theorem]{The extended residue for locally equivalent operators and Atiyah's $L^2$-index theorem}

In the following, denote by $(U;\mathcal E, \mathcal F)$ a triple where $U$ is a manifold and $\mathcal E,\mathcal F$ are bundles of finitely generated projective $\mathcal A$-Hilbert modules over $U$. Morphisms between these objects are of the form $
\alpha=(f; r,s)\colon (U';\mathcal E',\mathcal F')\to (X;\mathcal E,\mathcal F)$
where $f\colon U'\to U$ is an {\bf open embedding}, $r\in {\rm Hom} (\mathcal E', f^*\mathcal E)$, $s\in {\rm Hom} (f^*\mathcal F, \mathcal F')$.

Given any linear map $L\colon C_0^\infty (U, \mathcal E)\to C^\infty (U, \mathcal F)$, a morphism $\alpha \colon (U';\mathcal E',\mathcal F')\to (U;\mathcal E,\mathcal F)$ defines a map
$
\alpha^\sharp L \colon C_0^\infty (U', \mathcal E')\to C^\infty (U', \mathcal F')
$
that makes the following diagram commute
$$
\xymatrix{C_0^\infty (U, \mathcal E) \ar[r]^L & C^\infty (U, \mathcal F)\ar[d]^{\alpha^*}\\
C_0^\infty (U', \mathcal E') \ar[u]^{\alpha_*}\ar[r]^{\alpha^\sharp L} &C^\infty (U', \mathcal F')}
$$
where $\alpha^*$ denotes the composition of $s$ with the map induced by the pullback, and $\alpha_*$ denotes the composition of $r$ with the push-forward.
\begin{defn}
\label{def:loc.eq}
Let $M, M'$ be two manifolds and $A \in \Theta \left(M, {\mathcal E}\right), A'\in \Theta\left(M', {\mathcal E}'\right)$ be {\bf differential operators} acting on the sections of $\mathcal A$-Hilbert modules bundles  ${\mathcal E}, {\mathcal E}^\prime$ over $M$ and $M'$ respectively. The operators 
$A$ and $A'$ are said to be \emph{locally equivalent} if
\begin{itemize}
\item there exists a local diffeomorphism $\phi\colon M'\to M$ meaning that for any $ x'$ in $M'$ there is a neighborhood $U'$ of $x'$ such that $U=\phi(U')$ is open and $\phi_{U'}^U\colon U'\to U$ is a diffeomorphism
\item correspondingly, there are morphisms $\alpha=(f; r,s)\colon (U', {\mathcal E}'_{|U'}, {\mathcal E}'_{|U'})\to(U, {\mathcal E}_{|U}, {\mathcal E}_{|U})$
with $f=\phi_{U'}^U$, and with $r,s$ isomorphisms such that
\begin{equation}
\label{eq:phi}
A'_{U'}=\alpha^\sharp (A_U)
\end{equation}
where we have denoted by $A'_{U'}$ the restriction of $A'$ to the open set $U'$.
\end{itemize}
\end{defn}

\begin{ex}\label{ex:op.cov}(Twists by flat bundles of Hilbert modules)
Let $M$ be closed, $E\to M$ be a vector bundle, and $B\in \Theta(M, E)$ be a differential operator.  Let  $\mathcal F\to M$  be a flat bundle of finitely generated projective $\mathcal A$-Hilbert modules, endowed with a flat connection $\nabla_{\mathcal F}$. Denote by $\mathcal W$ the fibre. Let  $\underline{\mathcal W}=M\times \mathcal W$ denote the trivial bundle with fibre $\mathcal W$. Because $B$ is differential, one can consider on the one hand $B_{\W}$ to be the trivial extension of $B$ to $ E\otimes \W$ and on the other hand $B_{\mathcal F}$ the operator $B$ twisted by the flat connection on $\mathcal F$.  Then $A=B_{\W}$  on $\mathcal E'=E\otimes \W$ and $A'=B_{\mathcal F}$  on $\mathcal E'=E\otimes \mathcal F$ are locally equivalent by taking $\pi$ the identity map on $M$, and the local morphisms given by local trivialisations of $\mathcal F$ which are parallel with respect to $\nabla_{\mathcal F}$.
\end{ex}

\begin{ex} 
\label{ex:fond}As a particular case of the above (with $\mathcal W=\ell^2(\Gamma)$), let $\mathcal H=\widetilde M\times_\Gamma \ell^2(\Gamma)$ be the bundle defined in Example \ref{ex:cov}, and $\ellt=M\times \ell^2(\Gamma)$. If $A\in \Theta(M,E)$  is a differential operator, $A_{\mathcal H}$ is locally equivalent to $A_{\ellt}$.
\end{ex}  
\begin{ex}  Specialising to the torus and using the notations of 
	 Example \ref{ex:covTn}, for any differential operator  $A$ on $\T^n$,  the local equivalence between 
	 $A_{\mathcal H}$ and $A_{\ellt}$ translates to the well-known local equivalence of the
	 lifted operator $\pi^\sharp A$ (with a slight abuse of notation) on $\R^n$ with $A$.
	  \end{ex}
\begin{prop}
\label{prop:loc.equ} Let $A$, $A'$ and $B$, $B'$ be pairs of locally equivalent differential operators in the sense of Definition \ref{def:loc.eq}. Moreover assume that $B, B' $ are admissible. Then  with the notation of Definition \ref{def:loc.eq} 
\begin{eqnarray}
\label{eq:localrescov} 
f^*\left({\rm res}^\tau_x\left(A\log B\right)\, dx\right)&=& {\rm res}^\tau_{x'}\left(   A'\, \log \left(  B'\right)\right)\, dx', \\
 f^*\left({\rm res}^\tau_x\left(A\, B^\alpha\right)\, dx\right)&=& {\rm res}^\tau_{x'}\left(A'\,  \left( B'\right)^\alpha\right)\, dx'	\;\;,  \a\in \R
\end{eqnarray} for any point $x'$ and any local diffeomorphism $f=\phi_{U'}^U\colon   U' \to U$ on an open subset $U'$ containing $x'$,
which when integrated over  $M'$ yields
	\begin{eqnarray}\label{eq:reslogEQ}
&{\rm Res}^\tau\left(A \log B\right)={\rm Res}^\tau(A'  \log \left(  B'\right))\\
&	{\rm Res}^\tau\left(A  B^\alpha \right)={\rm Res}^\tau\left(A'  \left( B'\right)^\alpha \right)
	,\;\;  \a\in \R
\end{eqnarray}
	\end{prop}
\begin{proof}
 Since $B$ and $B'$ are admissible, there are smoothing operators $R, R'$  such  that   $Q=B+R$  and $Q=B'+R'$ are weights.  From Theorem \ref{thm:mainthmL2} applied to the operator $A$ in $\Theta(M,{\mathcal E})$, we know that the $n$-form $x\mapsto {\rm res}^\tau_x\left(A\log B\right)\, dx$  defines a global density so it transforms covariantly under coordinate transformations.
Thus, the pull-back of this density by the local diffeomorphism $f=\phi_{U'}^U\colon U'\to U$  reads

$$f^*\left({\rm res}^\tau_x\left(A\log B\right)\, dx\right)= {\rm res}^\tau_{ x'}(  f^\sharp A\, \log  f^\sharp B)\, dx' ,$$
where $x'=f^{-1}(x)$. Now $A'=f^\sharp A$ and $B'=f^\sharp B$, so we get the result.
\end{proof}
Specialising to (essentially) {\bf selfadjoint} elliptic differential operators provides an alternative proof of Atiyah's $L^2$-index theorem in the Hilbert module formulation.

\begin{cor}(Atiyah's $L^2$-index theorem)\label{cor:Aty}
 Let $D$ be an essentially selfadjoint differential operator of positive order $d$ acting on a $\Z_2$-graded vector bundle $E^+\oplus E^-\to M$, and assume $D$ is odd with respect to the grading, i.e. $
 D:= \begin{bmatrix}
 0 & D_-  \\
 D_+ & 0  
 \end{bmatrix}   $.
 Let $ D_\mathcal H:= \begin{bmatrix}
 0 & D_{\mathcal H,-}  \\
 D_{\mathcal H,+} & 0  
 \end{bmatrix}   $ be the twisted operator defined in Example \ref{ex:fond} acting on the bundle of $\mathcal N\Gamma$-Hilbert modules $E\otimes \mathcal H$. Then
 $$
 {\rm ind}^\tau(D_{\mathcal H, +})={\rm ind}( D_+)\ .
 $$
 \end{cor}

 \begin{proof} 
  Recall that by Example \ref{ex:op.cov} the operator $D_\mathcal H^2$ is locally equivalent to the trivial extension $D_{\ellt}$ acting on $\ellt=M\times \ell^2(\Gamma)$. Using Remark \ref{rk:inv}, $D_\mathcal H^2$ is admissible (with $\pi$ as a spectral cut), so we can apply Proposition \ref{prop:loc.equ} which yields \footnote{Prop. \ref{prop:loc.equ} easily extends to the $\Z_2$-graded case replacing the $\tau$-residue by the super $\tau$-residue.}
$$
{\rm sRes}^\tau(\log( D_{\mathcal H}^2))={\rm sRes}^\tau(\log( D_{\ellt}^2))\ .
$$
 By formula  \eqref{eq:indreslog}      
 $$
{\rm ind}^\tau ( D_{\mathcal H, +})=-\frac{1}{2d}{\rm sRes}^\tau(\log( D_{\mathcal H}^2))\ .
$$
Analogously,
$${\rm ind}(D_+)={\rm ind}^\tau ( D_{\ellt, +})=-\frac{1}{2d}{\rm sRes}^\tau(\log( D_{\ellt}^2))
$$
 where we have used that ${\rm ind}^\tau ( D_{\ellt,+})={\rm ind}(D_+)$, so that the equality follows. 
\end{proof} 
\begin{rk}
 This is a self-contained pseudodifferential proof of Atiyah's result, which relies on the locality of the Wodzicki residue. It is very similar in spirit to John Roe's proof on coverings \cite{Ro}. 
 \end{rk}
Let $D$ be an essentially selfadjoint differential operator acting on a vector bundle $E \to M$. Let $ D_\mathcal H$ be the twisted operator defined in Example \ref{ex:fond}.

Since  $D^2_{\mathcal H}$ is admissible (see Remark \ref{rk:inv}), the  operator   $Q_{\mathcal H}:=  D^2_{\mathcal H}+R$ defines a weight. 
 Let us consider  for any positive $t$ the associated heat-operator $e^{-tQ_{\mathcal H}}$. The corresponding heat-kernel  $\tau$-trace $K_t^\tau\left( Q_{\mathcal H} \right) (x,x)$ at a point $x$ is defined by 
 $${\rm Tr}^\tau\left(\phi\, e^{-tQ_{\mathcal H} }\right):=\int_M \phi(x)\, K_t^\tau(Q_{\mathcal H} )(x,x)\, dx\quad\forall \phi\in \Ci(M)\ .
 $$
Applying as in Section 2 an inverse Mellin transform to the holomorphic families $A(z)= \phi\, Q_{\mathcal H}^{-z}  $ and using Proposition \ref{prop:loc.equ}  we find that  
 \begin{eqnarray}
 \label{eq:HKL2} 
  && K_t^\tau\left( D^2_{\mathcal H} \right)(   x,  x) \\
 &\sim &-\frac{(4\pi )^{\frac{n}{2}}}{2 \sqrt{{\rm det}  g}(x) }
  \,\Bigg[ {\rm res}_x^\tau\left(\log  D^2_{\mathcal H}\right) \, \delta_{\frac{n}{2}-\left[\frac{n}{2}\right]} \nonumber \\
 	&+&
 	 \sum_{k\in\left[0, \frac{n}{2} \right[\cap \Z} \Gamma\left( \frac{n}{2}-k\right)  \,\ {\rm res}_x^\tau \left(   (D^2_{\mathcal H})^{k-\frac{n}{2}}\right) \, t^{k-\frac{n}{2}}\Bigg]\nonumber \\ 
  &\sim & -\frac{(4\pi )^{\frac{n}{2}}}{2 \sqrt{{\rm det}  g}(x) } 
  \,\Bigg[ {\rm res}_x \left(\log   D^2  \right) \, \delta_{\frac{n}{2}-\left[\frac{n}{2}\right]} \nonumber \\
 	&+&
 	 \sum_{k\in\left[0, \frac{n}{2} \right[\cap \Z} \Gamma\left( \frac{n}{2}-k\right)  \,\ {\rm res}_x  \left(   (  D^2 ) ^{k-\frac{n}{2}}\right) \, t^{k-\frac{n}{2}}\Bigg].  \nonumber
 	 \end{eqnarray}  
 \begin{rk} It follows from the Duhamel formula \cite{BGV} that the time zero asymptotics of $K_t^\tau\left( D^2_{\mathcal H} \right)(   x,  x) $ coincide with that of $K_t^\tau\left( D^2_{\mathcal H}+R \right)(   x,  x) $. This is here confirmed by the fact that the residues involved in the asymptotics are invariant under perturbation by the smoothing operator $R$.
 \end{rk}



\section{The scalar curvature on the noncommutative two-torus}  

We want to define the scalar curvature on the noncommutative two-torus by means of a Wodzicki residue in analogy to the formula  (\ref{eq:scalarcurvature}) established for Riemannian surfaces. 
We need a noncommutative analogue of Theorem \ref{thm:zetaresM} on the noncommutative torus $\T_\theta^n$.   Let us  briefly recall the results of \cite{LN-JP} we need for that purpose. 
\subsection[ Pseudodifferential operators on the noncommutative torus ]{ Pseudodifferential operators on the noncommutative torus }
 Let $\theta$ be  a symmetric $n\times n$ real matrix. The noncommutative deformation $\T^n_\theta$ of the commutative torus $\T^n\sim \R^n/\Z^n$ is encoded in the $C^*$-algebra ${A}_\theta$. An element $a\in {A}_\theta$ decomposes as the convergent series $a=\sum_{k\in \Z^n} a_k U_k$ where the $(U_k)$  are   unitaries in $A_\theta$ that satisfy $U_0=1$ and 
\[
 U_k U_{l} = e^{-2 \pi i\,\langle k,\th l\rangle}U_lU_k.
\]
Let ${\mathcal A}_\theta$ denote the algebra consisting of series of the form $\sum_{k\in \Z^n} a_k U_k$, where the sequence $(a_k)_{k}\in\mathcal{S}(\Z^n)$, the vector space of sequences $(a_k)_{k}$ that decay faster than the inverse of any polynomial in $k$.
We shall also need the linear form $\t$ on ${A}_\theta$ which to an element $a=\sum_{k\in \Z^n}a_k U_k $ assigns the scalar term $a_0$, and the Laplace operator $\mathbf{\Delta}=\sum_j \delta_j^2$ defined in \cite[Example 3.13]{LN-JP} acting on $A_\th$ with $\delta_j \left(\sum_{k\in \Z^n}a_k U_k\right)= 
\sum_{k\in \Z^n}k_j\,a_k U_k.$

We refer to \cite{LN-JP} for the construction of the corresponding algebra $\Psi(\T^n_\theta)$ of classical toroidal  pseudodifferential operators \cite[Paragraph 3.2]{LN-JP} on $\T^n_\theta$. When $\theta=0$, the noncommutative torus $\T^n_\theta$ coincides with $\T^n$, ${\mathcal A}_\theta$ with $\Ci(\T^n)$  and $\Psi(\T^n_\theta)$ with the algebra $\Psi(\T^n)$ of classical pseudodifferential operators on the closed manifold $\T^n$ considered in the first section. 

\subsection[ Holomorphic families of operators on the noncommutative torus ]{ Holomorphic families of operators on the noncommutative torus }

In \cite{LN-JP} we defined holomorphic families in $\Psi(\T^n_\theta)$ and extended the canonical trace to such families by $$ {\rm TR}_\theta (A(z)):=\cutoffsum_{\Z^n} \t\left(\Op_\theta^{-1}(A(z))\right)$$
where $\Op_\theta$ is the one to one map which takes a toroidal symbol to a toroidal operator. 
As seen in \cite[Proposition 6.2]{LN-JP}, the Wodzicki residue  ${\rm Res}_\theta $ on $\mathcal A_\theta$, is a noncommutative analogue of the classical Wodzicki residue and (up to a multiplicative factor) it is the only continuous linear form on $\Psi(\T^n_\theta)$  vanishing on smoothing operators  \cite{FW} (see also \cite{LN-JP} for a slightly different characterisation which does not require continuity). 
In contrast to this, the canonical trace ${\TR}_\theta $ is (up to a multiplicative factor) the only linear form on non integer operators in $\Psi(\T^n_\theta)$ whose restriction to trace-class operators is continuous \cite{LN-JP}. 

 \begin{thm}\label{thm:NCKVPSopNCT}   For any holomorphic family  $ A(z)\in \Psi(\T^n_\theta)$ of classical operators parametrised by $   \C$ with holomorphic order $   -qz+a$ for some positive $q$ and some real number $a$,
   \begin{enumerate}
 \item the map    $z\mapsto {\rm TR}_\theta\left(A(z)\right)  $  is  meromorphic with simple poles $d_j:=\frac{a+n-j}{q}$, $ j\in \Z_{\geq 0}$,
 \item  the complex  residue at the point $d_j$   is given by:
 \begin{equation}\label{eq:classicalKVnc}{\rm Res}_{z=d_j} {\rm TR}_\theta\left(A(z)\right)= \frac{1}{q}{ \rm
 Res}_\theta (A(d_j) ).
 \end{equation} 
 \item   If $A(d_j)$ is a differential operator, then $A^\prime (d_j)$ has a well defined 
Wodzicki residue  ${ \rm
  Res}_\theta (A^\prime(d_j) ) $     and we have  
  \begin{equation}\label{eq:PSclassicaloptheta}
{\rm fp}_{z=d_j}{\rm TR}\left(A  (z)\right)=  
 \frac{1}{q}{ \rm
 Res}_\theta(A^\prime (d_j)).
 \end{equation} 
 \end{enumerate} 
\end{thm}

 It was shown in \cite{LN-JP} how  one can define via  a Cauchy formula the logarithm \cite[Paragraph 7]{LN-JP} $\log\left(\mathbf{\Delta}\right)$ of the Laplace operator and  a noncommutative analogue  $\zeta_\theta(A,Q)$ \cite[Paragraph 7]{LN-JP}  of the  $Q$-regularised $\zeta$-trace of $A$ with $Q:=1+\mathbf{\Delta}$.   Applying Theorem \ref{thm:NCKVPSopNCT} to the holomorphic family $A(z)=A\,Q^{-z}$ yields the following extension of Theorem \ref{thm:zetaresM}  to the noncommutative torus. 
\begin{thm}\label{thm:NCzetares}
Let $Q:= 1+\mathbf{\Delta}$ and let $A$ be a {\rm differential operator} in $\Psi(\T^n_\theta)$. Then
\begin{enumerate}
\item the  $\zeta_\theta$-regularised  trace  $\zeta_\theta(A,Q)$ of $A$  is holomorphic at zero, 
\item  the residue ${\rm Res}_\theta$ extends to $A\log Q$  and we have
\begin{equation}\label{eq:zetathetaresAQ}
\zeta_\theta(A,Q)(0)=-\frac{1}{q}{\rm Res}_\theta(A\, \log Q).
\end{equation}
\end{enumerate}
\end{thm}

\subsection{ The scalar curvature as an extended Wodzicki residue }
We now want to define the scalar curvature on $\T^2_\theta$ by means of  a noncommutative analogue of (\ref{eq:scalarcurvature}). We work on a conformal deformation of the complexified two torus  using the notational conventions of \cite{FK2}.\\
Let $\tau=\tau_1+i\tau_2$ with $\tau_1,\tau_2\in \R$. Let $\partial=\delta_1+\bar \tau \delta_2$. Then $\partial^*=\delta_1+\tau \delta_2$.  The operators $\partial_\tau$ and $\partial_\tau^*$ are the noncommutative counterparts of   $-i\left(\partial_{x_1}+\bar \tau \partial_{x_2}\right)$ and $-i\left(\partial_{x_1}+ \tau \partial_{x_2}\right)$ acting on ${\mathcal A}_0=\Ci(\T^2)$.  Let $h\in  {\mathcal A}_\theta$ be selfadjoint and set $k=e^{\frac{h}{2}}$. Let ${\mathcal H}_h$ be the completion of ${\mathcal A}_\theta$ for the inner product $$ \langle  a,b\rangle_h= \t(b^*a k^{-2})= \langle  ak^{-1},bk^{-1}\rangle_0=\langle R_{k^{-2}} a,  b \rangle_0 $$ on ${\mathcal A}_\theta$ where $R_{b} : a\mapsto a\, b$ stand for the right multiplication by $b$.  The map $R_{k  } : a\mapsto a\, k  $ induces an isometry  $U: {\mathcal H}_0\longrightarrow {\mathcal H}_h$. 
\\   The analogue of the space of $ (1, 0)$-forms on the ordinary two-torus is defined to be the Hilbert space completion ${\mathcal H}_0^{(1,0)}$ of the space of finite sums $a\partial b, a, b \in {\mathcal A}_\theta$ for the inner product  $ \langle  a,b\rangle_0$.  We view $\partial$ as an unbounded operator $\partial_h : {\mathcal H}_h  \longrightarrow {\mathcal H}_0^{(1,0)}$. Then
$$\langle \partial  a, b\rangle_0= \langle  a   ,   \partial^*  b\rangle_0 =    \langle R_{k^{2}} a  ,    \partial^*  b\rangle_h=  \langle  a  ,   R_{k^{2}} \partial^*  b\rangle_h.$$  Thus   its formal adjoint is given by $$\partial_h^*=R_{k^{2}}\partial^* .$$  
We consider the  operator
\begin{displaymath}
D_h=\begin{bmatrix}
0 &\partial_h^* \\
\partial_h &0
\end{bmatrix} =\begin{bmatrix}
0 &R_{k^{2}}\partial^* \\
\partial  &0
\end{bmatrix} 
\end{displaymath}   
acting on the $\Z_2$- graded space  $\widetilde{{\mathcal H}_h}:= {\mathcal H}_h\oplus {\mathcal H}_0^{(1,0)} $.
Let $$ \mathbf{\Delta}_h:=\partial_h^*\partial_h+ \partial_h\partial_h^*= R_{k^{2}}\,\partial^*\partial +\partial R_{k^{2}}\partial^*,$$ which for $\tau_i$ and $h=0$ coincides with $\mathbf{\Delta}$. \\
Here is a corollary of  Theorem \ref{thm:zetaresM} extended to this slightly more general framework. 

\begin{cor}\label{thm:zetaresTtheta}  Let $ Q:=  \mathbf{\Delta}_h+\pi_{\mathbf{\Delta}_h}$, where   $\pi_{\mathbf{\Delta}_h}$ is the orthogonal projection onto the kernel of $\mathbf{\Delta}$  and let $a\in {\mathcal A}_\theta$. Then
\begin{enumerate}
\item the  $\zeta_\theta$-regularised  trace  $\zeta_\theta(a,Q)$ of $a$  is holomorphic at zero, 
\item  the residue ${\rm Res}_\theta$ extends to $a\log Q$  and we have
\begin{equation}\label{eq:zetaresaQ}
\zeta_\theta(a,Q)(0)=-\frac{1}{q}{\rm Res}_\theta(a\, \log Q).
\end{equation}
\end{enumerate}
\end{cor}
 
 Exactly as in the case of manifolds, if we specialise the holomorphic family to the case $A(z)=a \, Q^{-z}$, where $a$ is an element of the algebra of the noncommutative torus, using the same arguments, we get the following result:
 
\begin{thm}\label{thm:NCan}  
For any $a\in \A_\th$ we have
\begin{eqnarray}\label{eq:NCHKresphi}   
&& {\rm Tr}\left(a\, e^{-t \mathbf{\Delta}_h }\right) \sim_{t\to 0} \\
&&-\frac{(4\pi )^{\frac{n}{2}}}{2 }\,\Bigg[{\rm Res} \left(a\,\log \mathbf{\Delta}_h\right) \, \delta_{\frac{n}{2}-\left[\frac{n}{2}\right]} \nonumber \\
&&+\sum_{k\in\left[0,\frac{n}{2}\right[\cap \Z} \Gamma\left( \frac{n}{2}-k\right)  \,{\rm Res} \left(a\,    \mathbf{\Delta}_h^{k-\frac{n}{2}}\right)\, t^{k-\frac{n}{2}}\Bigg]\nonumber.
\end{eqnarray}
\end{thm}

The scalar curvature being associated to the $a_1$ coefficient of the heat kernel expansion, Theorem \ref{thm:NCan}  motivates the following definition:

\begin{defn}
The ``scalar curvature" $\mathfrak s_h$ on the noncommutative two torus $\T^2_\theta$ associated with the ``metric" determined by the conformal factor  $h$ is defined as 
\begin{equation}\label{eq:NCscal}
\langle \mathfrak s_h, a\rangle_h 
= \begin{cases}
-6\pi\,{\rm Res}_\theta \left(a\,\log \mathbf{\Delta}_h\right)\quad  &\text{if}\quad n=2\noindent\\
 - \frac{ 3}{2}\, (4\pi )^{\frac{n}{2}} \, \Gamma\left( \frac{n}{2}-k\right)\, 
{\rm Res}_\theta \left(a\, \mathbf{\Delta}_h^{k-\frac{n}{2}} \right)\quad &{\rm otherwise},
\end{cases} 
\end{equation} 
 which compares with the definitions in  \cite{CM2}, \cite{CT}, \cite{FK1}, \cite{FK2}.
\end{defn}


\subsection*{Acknowledgment}
The last author thanks the organisers of the    XXXIII Workshop on
Geometric Methods in Physics for giving her the opportunity to present some of these results during the meeting.

\bigskip

\noindent Sara Azzali\\
Institut f\"ur Mathematik\\
Universit\"at Potsdam\\
Am Neuen Palais, 10\\
14469 Potsdam, Germany\\
e-mail: \texttt{azzali@uni-potsdam.de}

\bigskip

\noindent Cyril L\'evy\\
 D\'epartement de math\'ematiques\\
Centre universitaire Jean-Fran\c cois Champollion\\
Place Verdun\\
81000 Albi, France\\
e-mail: \texttt{cyril.olivier.levy@gmail.com}

\bigskip

\noindent Carolina Neira-Jim\'enez\\
Departamento de Matem\'aticas\\
Universidad Nacional de Colombia\\
Carrera 30 \# 45-03\\
Bogot\'a, Colombia\\
e-mail: \texttt{cneiraj@unal.edu.co}

\bigskip

\noindent Sylvie Paycha\\
Institut f\"ur Mathematik\\
Universit\"at Potsdam\\
Institut f\"ur Mathematik\\
Am Neuen Palais, 10\\
14469 Potsdam, Germany\\
e-mail: \texttt{paycha@math.uni-potsdam.de}

\end{document}